\documentclass[12pt, a4paper, leqno]{amsart}
\usepackage[utf8]{inputenc}
\usepackage{amsmath}
\usepackage{amsfonts}
\usepackage{amssymb}
\usepackage{times}
\usepackage{enumitem}
\usepackage[T1]{fontenc} 
\usepackage{url} 
\usepackage{color,esint}
\usepackage[colorlinks, linkcolor=blue, citecolor=blue]{hyperref} 
\usepackage{graphicx} 
\usepackage{enumitem}
\usepackage{kotex}
\usepackage{tabularx}
    \usepackage{mathtools}
\usepackage{comment}

\let\oldmarginpar\marginpar
\renewcommand\marginpar[1]{\-\oldmarginpar[\raggedleft\footnotesize #1]%
{\raggedright\footnotesize #1}}

\usepackage{amsthm}
\theoremstyle{plain}
\newtheorem{thm}[equation]{Theorem}
\newtheorem{lem}[equation]{Lemma}
\newtheorem{prop}[equation]{Proposition}

\theoremstyle{definition}

\theoremstyle{remark}
\newtheorem{rem}[equation]{Remark}

\numberwithin{equation}{section}

\newcommand{\R}{\mathbb{R}}

\renewcommand{\div}{\divop}
\def\le{\leqslant}
\def\leq{\leqslant}
\def\ge{\geqslant}
\def\geq{\geqslant}
\def\phi{\varphi}
\def\rho{\varrho}
\def\vartheta{\theta}

\def\div{\qopname\relax o{div}}

\date{\today}

\begin{document}

\title[]{{Double phase quasiconvex functionals and their partial regularity theory}
}

\author{Sunwoo Jeong}
  \address{Sunwoo Jeong, Department of Mathematics, Institute for Mathematical and Data Science, Sogang University, Seoul 04107, Republic of Korea}
\email{sunwoo@sogang.ac.kr}

\author{Jihoon Ok}
\address{Jihoon Ok, Department of Mathematics, Institute for Mathematical and Data Science, Sogang University, Seoul 04107, Republic of Korea}
\email{\texttt{jihoonok@sogang.ac.kr}}

\thanks{MSC(2020) 35J50, 35J92, 35B65, 46E30}

\subjclass[2020]{} 

\begin{abstract} 
We consider degenerate nonautonomous energies
$$
\int_\Omega f(x, Dv)\, dx,
$$
for vector-valued functions $v \in W^{1,1}(\Omega, \mathbb{R}^N)$, where the integrand $f(x,P)$ satisfies growth and weak uniform   quasiconvexity assumption associated with the double phase function $H(x,t)=t^p + a(x)t^q$. We establish partial Hölder regularity for the gradients of minimizers  under suitable, and possibly minimal, regularity assumptions on $H$ and $f$. Our approach relies on two approximation results: $\mathcal{A}$-harmonic approximation and a variational version of the $\phi$-harmonic approximation.
\end{abstract}

\maketitle


\section{Introduction}

The aim of this paper is to study the regularity theory of minimizers of the functional
\begin{equation}\label{F}
\mathcal{F}(v) = \int_\Omega f(x, Dv(x))\, dx,
\end{equation}
where $\Omega$ is a bounded open subset of $\mathbb{R}^n$ with $n \ge 2$. The integrand $f$ is a possibly degenerate Carathéodory function exhibiting non-standard growth and quasiconvexity. The admissible functions $v$ belong to $W^{1,1}(\Omega, \mathbb{R}^N)$ with $N \ge 1$, for which the energy $\mathcal{F}(v)$ is well defined. In particular, we focus on integrands with double phase type growth and quasiconvexity assumptions.

The development of partial regularity theory for nonlinear elliptic systems and variational integrals has its origins in the methods of Geometric Measure Theory. These methods were first introduced by De Giorgi and Almgren in their studies of minimal surfaces and minimizing varifolds, and later advanced through the classical works of Morrey, Giusti, and Miranda \cite{M52,M67,GM68}. After a few decades, Duzaar and Steffen \cite{DS02} provided a more direct proof of regularity for minimizers of elliptic integrals in Geometric Measure Theory, based on the $\mathcal{A}$-harmonic approximation method. This approach, which traces back to the ideas of De Giorgi and was later developed further by Simon \cite{S83,S96}, simplifies many technical arguments and allows one to obtain optimal regularity results, including those for boundary value problems. In the same line of research, the method was subsequently applied to the parametric case in \cite{DG00} and \cite{DGG00}, where it yielded optimal regularity results for nondegenerate elliptic systems and almost minimizers of quasiconvex integrals, providing a clear and effective framework for regularity theory that also extends to boundary problems (see \cite{G02}). The main analytical tool in this framework is a version of the $\mathcal{A}$-harmonic approximation lemma \cite{DS02} (see Lemma \ref{Aharmonicapproximation} below), which roughly states that if a map $w$ approximately satisfies a linear elliptic system with constant coefficients in the sense of (\ref{eq:almostAsense}), then one can find an exact solution $h$ to the same system that is close to $w$ in the sense of (\ref{eq:Aapproximatesense}).

The regularity theory for nonlinear elliptic systems and variational integrals began with scalar equations exhibiting degenerate or singular behavior, obtaining local H\"older continuity of the gradient. In the vectorial case, examples \cite{N78,SY02} suggest that only partial regularity of minimizers of (\ref{F}) is generally expected unless Uhlenbeck structure is assumed in $f$, that is, unless $f$ is isotropic in the latter term. Via the blow-up techniques that appeared first in \cite{G61,A68}, solutions to systems of the type $-\div(a(x)|Du|^{p-2}Du)=\mu$ are shown, for instance in \cite{U77,KM18}, to have everywhere regularity, provided that $a(x)$ and $\mu$ are regular enough and partial regularity in general quasiconvex problems with $p$-growth, which we refer to \cite{E86,AF87,DK02}. These results, however, consider nondegenerate autonomous integrands $f$. Using a version of $\mathcal{A}$-harmonic and $p$-harmonic approximations, Duzaar and Mingione proved, for degenerate quasiconvex functionals with $p$-growth and autonomous integrand $f(P)\equiv f(x,P)$, partial $C^{1,\alpha}$-regularity under the additional assumption that $\partial f(P)\to |P|^{p-2}P$ as $P\to\mathbf{0}$ formally \cite{DM041}. Note that the $\mathcal{A}$-harmonic approximation was first employed in partial regularity theory by Duzaar and Grotowski \cite{DG00}, while the $p$-harmonic approximation was first obtained in \cite{DM04}. Moreover, Diening, Stroffolini and Verde obtained improved approximation results using the Lipschitz truncation method which allows us to consider general $\phi$-growth problems \cite{DLSV12,DSV12}. Foss and Mingione investigated for nonautonomous quasiconvex funcionals with $p$-growth in \cite{FM08} and see also \cite{B12} for cases where degeneracy is further assumed. 
We further mention that partial regularity results for nonautonomous elliptic systems or quasiconvex functionals with non-standard growth are obtained in \cite{H13,G16,O16,GPRT17,GPT17,O18,O18B,D19,CO20}.

The double phase energy
\begin{equation*}
    \int_{\Omega}H(x,|Dv|)\, dx, \quad \text{with }\ H(x,t)=t^p+a(x)t^q,
\end{equation*}
where $1<p\le q$ and $0\le a(x) \le L$, was first introduced by Zhikov \cite{Z86,Z95,Z97} in the setting of homogenization and as a prototype exhibiting the Lavrentiev phenomenon. Esposito, Leonetti and Mingione \cite{ELM04} introduced the optimal condition \eqref{a} for $H$ to obtain regularity results for minimizers of the above functional, providing further examples of double phase problems linked to the Lavrentiev phenomenon, and proved higher integrability for the minimizers. We also refer to \cite{FMM04} for related results. Along the same lines, a sharp and maximal regularity result that the minimizer $u$ satisfies $Du\in C_{loc}^{0,\gamma}(\Omega)$ for some $\gamma\in(0,1)$ has been obtained under the condition (\ref{a}) in \cite{CM15.1,CM15.2,BCM18}. Furthermore, partial regularity for nondegenerate systems with double phase growth were addressed in \cite{O18,O18B,SS23} assuming the superquadratic range ($p\geq2$ ). General degenerate systems with double phase growth was investigated in \cite{OSS25} where partial H\"older continuity of the gradient of their weak solutions was established, without the superquadratic condition for $p$.

In this paper, we investigate the general degenerate energy functional \eqref{F} with double phase growth and quasiconvexity assumptions, and we establish a partial $C^{1,\alpha}$-regularity result for its minimizer $u$. Our overall strategy follows that of \cite{CO20}, which combines the approaches developed in \cite{DM041,DLSV12}. However, the double phase quasiconvex setting leads to various analytical difficulties that require new ideas.

We first obtain self-improving estimates, also known as higher integrability estimates, for both $H(\cdot,|Du|)$ and $H_{|Q|}(\cdot,|Du-Q|)$, where $H_{|Q|}$ denotes the shifted function associated with $H$ for $Q\in\mathbb R^{N\times n}$ (see \eqref{def:shift} with $\phi(t)=H(x,t)$ and $a=|Q|$). We then distinguish two regimes. The nondegenerate regime is characterized by
\[
\fint_{B_{r}(x_0)}
\left|V_{H_{B_{r}(x_0)}^-}(Du)-
\left(V_{H_{B_{r}(x_0)}^-}(Du)\right)_{x_0,r}\right|^2 dx
\ll
\fint_{B_{r}(x_0)}\left|V_{H_{B_{r}(x_0)}^-}(Du)\right|^2 dx,
\]
while the complementary case corresponds to the degenerate regime. Here, $H^-_{B_r}(t)$ denotes the infimum of $H(x,t)$ over $x\in B_r$, and $V_\phi$ is defined in \eqref{def:V}. In each regime, we derive an excess-decay estimate for $V_{H^-_{B_r}}(Du)$: in the nondegenerate regime this relies on the $\mathcal A$-harmonic approximation lemma (Lemma~\ref{Aharmonicapproximation}), whereas in the degenerate regime it is obtained by applying the $\phi$-minimizing approximation lemma (Lemma~\ref{lem:minimizingapproximation}).

 We summarize the main issues and novelties of our approach below.

\begin{itemize}

\item In the degenerate regime, existing partial regularity results for quasiconvex functionals with $p$- or Orlicz-type growth rely on the $p$- or $\phi$-harmonic approximation together with Ekeland’s variational principle; see \cite{DM041,CO20}. However, it is unclear whether Ekeland’s principle is applicable to the double phase quasiconvex setting (see \eqref{C4}). Moreover, under this quasiconvex setting, it is also unclear that the higher integrability $Du\in L^{q}_{\mathrm{loc}} (\Omega,\R^{N\times n})$ holds.  We therefore adopt a more direct approach based on the $\phi$-minimizing approximation lemma, which is a variational analogue of the $\phi$-harmonic approximation, which enables a direct comparison between $Du$ and the gradient of a minimizer of an isotropic functional.

\item To apply the $\phi$-minimizing approximation lemma in the degenerate regime, we impose condition \eqref{C7} on the integrand $f$, which ensures that $f(x,P)-f(x,\mathbf 0)$ approximates $H(x,|P|)$ as $P\to \mathbf 0$. In previous works, a similar assumption has been required for $\partial f(x,P)=D_P f(x,P)$. In Remark~\ref{rem:C7}, we show that our assumption \eqref{C7} is more general.

\item We assume a Hölder continuity condition in $x$ for $f(x,P)$ itself (see \eqref{C5}), rather than for $\partial f(x,P)$. This condition is essential for establishing the Caccioppoli inequality in the shifted setting (Lemma~\ref{lem:Caccioppoli2}) and for the $\phi$-minimizing approximation. Combined with the growth condition \eqref{C3} as well as \eqref{C1}, it also yields a Hölder continuity condition for $\partial f(x,P)$ (Lemma~\ref{lem:(A4.5)}), which is needed in the nondegenerate regime.

\item To apply the $\mathcal A$-harmonic approximation to $Du - Q$ with the shifted function $H_{|Q|}$ in the nondegenerate regime, we first establish self-improving estimates for $H_{|Q|}(|Du-Q|)$. This step is substantially more delicate in the double phase quasiconvex setting than in the classical $p$- or Orlicz-growth cases, and it requires both the Hölder continuity assumption \eqref{C5} for $f(x,P)$ and the self-improving estimate for $Du$.

\end{itemize}

\subsection{Main results}
 
Let $H:\Omega\times[0,\infty)\rightarrow[0,\infty)$ be defined by
\begin{equation}\label{H}
    H(x,t):=t^p+a(x)t^q,
\end{equation}
with $a:\Omega\rightarrow [0,\infty)$ and $1<p\leq q$ satisfying
\begin{equation}\label{a}
    a\in C^{0,\alpha}(\Omega),\quad0\leq a\leq L,\quad \frac{q}{p}\leq1+\frac{\alpha}{n},
\end{equation}
for some $\alpha\in(0,1]$ and $L>0$. We denote by $H'(x,t)$ and $H''(x,t)$ the first and the second derivatives of $t\mapsto H(x,t)$, respectively. We can see that $H(x,t)=t^p$ exhibits $p$-phase when $a(x)=0$, whereas if $a(x)>0$, $H$ is of $(p,q)$-phase.

With $H$, we consider a Caratheodory function $f:\Omega\times\mathbb{R}^{N\times n}\rightarrow\mathbb{R}$ that satisfies the following properties with constants  $0\le \nu \le L$ and $\beta_0\in(0,1)$: denoting the gradient and the Hessian of $f$ with respect to $P$ by $\partial f(x,P) \in \R^{N\times n}$ and  $\partial^2 f(x,P) \in \R^{(N\times n)^2}$, respectively,
\begin{enumerate} [label=(C\arabic*), ref=C\arabic*] 
\item \label{C1}
\textbf{differentiability:} for every $x\in\Omega$, the function
\begin{equation*}
    P\in\mathbb{R}^{N\times n}\mapsto f(x,P)
\end{equation*}
is of class $C^1(\mathbb{R}^{N\times n})\cap C^2(\mathbb{R}^{N\times n}\setminus\{{\bf 0}\})$;

\item \label{C2}
\textbf{coercivity:} $f(\cdot,\mathbf{0})$ is integrable on $\Omega$ and  for every $x\in\Omega$ and $P\in\mathbb{R}^{N\times n}$,
\begin{equation*}
    \nu H(x,|P|)\leq f(x,P)-f(x,\mathbf{0});
\end{equation*}

\item \label{C3} 
\textbf{growth condition:} for every $x\in\Omega$ and $P\in\mathbb{R}^{N\times n}\setminus\{{\bf0}\}$, there holds
\begin{equation}\label{ass:gc}
\begin{cases}
    f(x,P)-f(x,\mathbf{0})\leq L H(x,|P|)\\
    |\partial^2 f(x,P)|\leq L H''(x,|P|).;
\end{cases}
\end{equation}

\item \label{C4}
\textbf{uniform $W^{1,H}$-quasiconvexity:}  for every $y\in\Omega$, $P\in\mathbb{R}^{N\times n}$ and open ball $B\subset\mathbb{R}^n$,
\begin{equation*}
    \int_Bf(y,P+D\zeta(x))-f(y,P)\, dx\geq \nu \int_BH''(y,|P|+|D\zeta(x)|)|D\zeta(x)|^2\, dx
\end{equation*}
holds for all $\zeta\in C_c^\infty(B,\mathbb{R}^N)$;
\item \label{C5}

\textbf{H\"older continuity assumption with respect to $x$:} 
for every $x_1,x_2\in\Omega$ and every $P\in\mathbb{R}^{N\times n}$,
\begin{equation*}
    |f(x_1,P)-f(x_2,P)|\leq L|x_1-x_2|^{\beta_0}\big(H(x_1,|P|)+H(x_2,|P|)\big)+L|a(x_1)-a(x_2)||P|^q;
\end{equation*}

\item \label{C6}
 \textbf{assumption for the nondegenerate case:} 
 \begin{equation*}
    |\partial^2 f(x,P)-\partial^2 f(x,P+Q)|\leq c_1H''(x,|P|)\left(\frac{|Q|}{|P|}\right)^{\beta_0}
\end{equation*}
holds for every $x\in\Omega$ and $P,Q\in\mathbb{R}^{N\times n}$ with $0<|Q|\leq\frac{|P|}{2}$;

\item \label{C7}
\textbf{assumption for the degenerate case:} for every $x\in\Omega$ and $\delta>0$, there is $\kappa=\kappa(\delta)>0$ such that whenever $P\in \R^{N\times n}$ with $0< |P| < \kappa$, there holds
\begin{equation}\label{ass:A6}
    |H(x,|P|)-f(x,P)+f(x,\mathbf{0})|<\delta H(x,|P|).
\end{equation}
\end{enumerate}
Note that the condition \eqref{C4} implies the following strong Legendre-Hadamard condition of $\partial^2 f(x,P)$: for every $x\in \Omega$,  $P\in\mathbb{R}^{N\times n}\setminus\{{\bf0}\}$, $b=(b^1,\dots,b^N)\in\mathbb{R}^N$, and  $z\in\mathbb{R}^n$
\begin{equation}\label{eq:LH}
    \langle\,  \partial^2 f(x,P) (b\otimes z)\, |\, b\otimes z \, \rangle\geq cH''(x,|P|)|b|^2|z|^2,
\end{equation}
where $\langle\, \cdot\, |\, \cdot\, \rangle$ is the inner product in $\mathbb{R}^{N\times n}$ and $b\otimes z:=(b^iz_j)_{i,j}\in\mathbb{R}^{N\times n}$ denotes their tensor product, as we follow the argument in \cite{Giusti}.  

\begin{rem}\label{rem:C7}
In previous literature, such as \cite{DM04, CO20}, assumptions for the degenerate case are typically formulated using the derivative $\partial f$. In our context, an analogous condition can be stated as:
\begin{equation*}
    \left| H'(x, |P|) \frac{P}{|P|} - \partial f(x, P) \right| < \delta H'(x, |P|),
\end{equation*}
where the parameters $x, \delta,$ and $P$ are as specified in \eqref{C7}. This derivative-based approach is more restrictive than our assumption \eqref{C7}. Indeed, the former implies the latter through the following integral estimate:
\begin{equation*}
    \begin{aligned}
        |H(x, |P|) - f(x, P) + f(x, \mathbf{0})| 
        &= \left| \int_0^1 \frac{d}{dt} \left[ H(x, t|P|) - f(x, tP) \right] dt \right| \\
        &= \left| \int_0^1 \left\langle H'(x, t|P|) \frac{P}{|P|} - \partial f(x, tP) \, \middle| \, P \right\rangle dt \right| \\
        &\leq \int_0^1 \left| H'(x, t|P|) \frac{P}{|P|} - \partial f(x, tP) \right| |P| \, dt \\
        &< \int_0^1 \delta H'(x, t|P|) |P| \, dt = \delta H(x, |P|).
    \end{aligned}
\end{equation*}
Consequently, our assumption \eqref{C7} represents a refinement and relaxation of the conditions used in previous works.
\end{rem}

Our main theorem is stated as follows.
 \begin{thm}\label{mnthm}
     Let $H\,:\,\Omega\times[0,\infty)\to[0,\infty)$ be defined as in (\ref{H}) with (\ref{a}), and $f:\Omega\times\mathbb{R}^{N\times n}\to\mathbb{R}$ satisfy the assumptions  \eqref{C1}--\eqref{C7}. If $u\in W^{1,1}(\Omega,\mathbb{R}^N)$ with $H(\cdot,|Du|)\in L^1(\Omega)$ minimizes $\mathcal{F}$, then there exist $\beta=\beta(n,N,p,q,\alpha,L,[a]_{C^{0,\alpha}},\nu,\beta_0)\in(0,1)$ and an open subset $\Omega_0\subset\Omega$ such that $|\Omega\setminus\Omega_0|=0$ and
     \begin{equation*}
         Du\in C_{loc}^{0,\beta}\left(\Omega_0,\R^{N\times n}\right).
     \end{equation*}
     Moreover, $\Omega\setminus\Omega_0\subset\Sigma_1\cup\Sigma_2$, where
     \begin{equation*}
         \begin{aligned}
             &\Sigma_1:=\left\{x_0\in\Omega\ \Bigg|\ \liminf_{r\to0+}\fint_{B_r(x_0)}|V_{H_{B_r(x_0)}^-}(Du)-(V_{H_{B_r(x_0)}^-}(Du))_{x_0,r}|^2\, dx>0\right\},\\
             &\Sigma_2:=\left\{x_0\in\Omega\ \Bigg|\ \limsup_{r\to0+}\fint_{B_r(x_0)}|V_{H_{B_r(x_0)}^-}(Du)|^2\, dx=\infty\right\}.
         \end{aligned}
     \end{equation*}
 \end{thm}

 \begin{rem}
     In view of (\ref{eq:frozenH}), (\ref{eq:Vtoshift}) and (\ref{eq:Vmean}),
     \begin{equation*}
         \fint_{B_r(x_0)}\left|V_{H_{B_r(x_0)}^-}(Du)\right|^2\, dx\leq \fint_{B_r(x_0)}\left|V_{H}(x,Du)\right|^2\, dx
     \end{equation*}
     and
     \begin{equation*}
     \begin{aligned}
         \fint_{B_r(x_0)}|V_{H_{B_r(x_0)}^-}(Du)-(V_{H_{B_r(x_0)}^-}&(Du))_{x_0,r}|^2\, dx\\
         &\lesssim\fint_{B_r(x_0)}|V_{H}(x,Du)-(V_{H}(x,Du))_{x_0,r}|^2\, dx.
     \end{aligned}
     \end{equation*}
 Thus, since $|V_H(x,Du)|^2\sim H(x,Du)\in L^1(\Omega)$, we deduce from the Lebesgue differentiation theorem that $|\Sigma_1|=|\Sigma_2|=0$.
 \end{rem}


\section{Preliminaries}

\subsection{Basic notation}\ \\
$\Omega$ is an open bounded domain in $\mathbb{R}^n$ with $n\ge2$, and $B_r(x_0)$ denotes the open ball of radius $r>0$ centered at $x_0\in\mathbb{R}^n$. $B_r$ may often be used in place of $B_r(0)$, in the case when $x_0=0$. For $g\in L^1(B_r(x_0),\mathbb{R}^m)$, the average of $g$ in $B_r(x_0)$ is denoted as
\begin{equation*}
    (g)_{x_0,r}:=\fint_{B_r(x_0)}g\, dx.
\end{equation*}
$\mathbb{R}^{N\times n}$ is the space of all $N\times n$ real matrices.  We let $Dw$ denote the distributional derivative of $w\in L^1(\Omega,\mathbb{R}^N)$ in $\mathbb{R}^{N\times n}$. For $p>1$, $p':=\frac{p}{p-1}$ represents the H\"older conjugate of $p$. Moreover, if $1\le p<n$, $p^*:=\frac{np}{n-p}$ stands for the Sobolev conjugate of $p$ and if $p\geq n$, $p^*$ is any real number larger than $p$. Given $U\subset\mathbb{R}^n$, $\chi_U$ is the characteristic function with respect to $U$. For two real-valued functions $g$ and $h$ on the same domain, write $g\lesssim h$ provided that $g\leq ch$ for some $c>0$ only depending on structural constants and write $g\sim h$ if $g\lesssim h$ and $h\lesssim g$. Throughout the paper, we denote by $u$ a minimizer of (\ref{F})

\subsection{N-functions}\ \\
Here we recall some elementary definitions and basic results about Orlicz functions.

 An \textit{N-function} is defined as a real-valued function $\varphi:[0,\infty)\to[0,\infty)$, which is convex and satisfies the following conditions: 
\begin{itemize}
    \item $\varphi(0)=0$
    \item $\varphi$ admits the right-continuous, nondecreasing derivative $\varphi'$ which satisfies $\varphi'(0)=0$
    \item $\varphi'(t)>0$ for $t>0$, and $\lim_{t\to+\infty}\varphi'(t)=\infty$.
\end{itemize}
$\varphi$ is said to satisfy the $\Delta_2$-condition and write $\varphi\in\Delta_2$ if there is $c>0$ such that for all $t\geq0$, there holds $\varphi(2t)\leq c\varphi(t)$. The smallest such $c$ is denoted by $\Delta_2(\varphi)$. 

For an N-function $\varphi$, we assume that
\begin{equation}\label{eq:Nincrange}
    p_1\leq\inf_{t>0}\frac{t\varphi'(t)}{\varphi(t)}\leq\sup_{t>0}\frac{t\varphi'(t)}{\varphi(t)}\leq p_2,
\end{equation}
for some $1<p_1\leq p_2<\infty$. Furthermore, we can also assume that $\varphi\in C^2((0,\infty))$ satisfies
\begin{equation}\label{eq:N'incrange}
    0<p_1-1\leq\inf_{t>0}\frac{t\varphi''(t)}{\varphi'(t)}\leq\sup_{t>0}\frac{t\varphi''(t)}{\varphi'(t)}\leq p_2-1.
\end{equation}
Note that if $\varphi$ satisfies (\ref{eq:N'incrange}), then (\ref{eq:Nincrange}) holds and hence we have
\begin{equation}\label{eq:Neqv}
    \varphi(t)\sim t\varphi'(t)\quad\text{ and }\quad \varphi(t)\sim t^2\varphi''(t),\qquad t>0.
\end{equation}
For $H$ defined as in (\ref{H}) and fixed $x\in\Omega$, $\varphi(t):=H(x,t)$ is an N-function satisfying (\ref{eq:N'incrange}) with $p_1=p$ and $p_2=q$.

The conjugate of $\varphi$ is denoted by $\varphi^*(t):=\sup_{s\geq0}(st-\varphi(s))$. Then it is again an N-function with $\frac{p_2}{p_2-1}$ and $\frac{p_1}{p_1-1}$ in place of $p_1$ and $p_2$, respectively, and $(\varphi^*)^*=\varphi$. $\Delta_2(\varphi,\varphi^*)$ shall denote the constants depending on $\Delta_2(\varphi)$ and $\Delta_2(\varphi^*)$. The following properties then follow from \eqref{eq:Nincrange}. 

\begin{prop}
    Let $\varphi:[0,\infty)\to[0,\infty)$ be an N-function with (\ref{eq:Nincrange}). Then
    \begin{itemize}
        \item[$(a)$] the mappings
        \begin{equation*}
            t\in(0,\infty)\to\frac{\varphi'(t)}{t^{p_1-1}},\,\frac{\varphi(t)}{t^{p_1}} \quad\text{ and }\quad t\in(0,\infty)\to\frac{\varphi'(t)}{t^{p_2-1}},\,\frac{\varphi(t)}{t^{p_2}}
        \end{equation*}
        are nondecreasing and nonincreasing, respectively. In particular,
        \begin{equation}\label{eq:Nconstant}
                \min\{C^{p_1},C^{p_2}\}\varphi(t)\leq  \varphi(Ct) \le  \max\{C^{p_1},C^{p_2}\}\varphi(t)  ,\quad C>0.
        \end{equation}
        Moreover,
        \begin{equation*}
             \min\{C^{\frac{p_1}{p_1-1}},C^{\frac{p_2}{p_2-1}}\} \varphi^*(t) \le  \varphi^*(Ct) \le  \max\{C^{\frac{p_1}{p_1-1}}, C^{\frac{p_2}{p_2-1}}\}\varphi^*(t),\quad C>0.
        \end{equation*}
        \item[$(b)$] (Young's inequality) for any $\lambda\in(0,1]$, it holds that
        \begin{equation}\label{eq:Nyoung}
                \tau t\leq\lambda^{-p_2+1}\varphi(\tau)+\lambda\varphi^*(t) 
                \quad\text{and}\quad
                \tau t\leq \lambda\varphi(\tau)+\lambda^{-\frac{1}{p_1-1}}\varphi^*(t).
        \end{equation}
        \item[$(c)$] there exists a constant $c=c(p_1,p_2)>1$ such that
        \begin{equation}\label{eq:conjugatecompositedifferentiation}
            c^{-1}\varphi(t)\leq\varphi^*(t^{-1}\varphi(t))\leq c\varphi(t).
        \end{equation}
    \end{itemize}
\end{prop}

Note from (\ref{eq:Nconstant}) that $\varphi$ and $\varphi^*$ both satisfy $\Delta_2$-condition with constants $\Delta_2(\varphi)$ and $\Delta_2(\varphi^*)$ determined by $p_1$ and $p_2$.

Also from \cite[Lemma 20]{DE08}, we have the following property for N-functions.
\begin{lem}\label{lem:phi''line}
    Let $\varphi$ be an N-function with $\varphi,\varphi^*\in\Delta_2$. Then for all $P_0,P_1\in\mathbb{R}^{N\times n}$ with $|P_0|+|P_1|>0$, there holds
    \begin{equation*}
        \int_0^1\frac{\varphi'(|(1-\theta)P_0+\theta P_1|)}{|(1-\theta)P_0+\theta P_1|}\, d\theta\sim\frac{\varphi'(|P_0|+|P_1|)}{|P_0|+|P_1|},
    \end{equation*}
    where hidden constants depend only on $\Delta_2(\varphi,\varphi^*)$.
\end{lem}

 We define the Orlicz and Orlicz--Sobolev spaces, denoted by 
$L^\varphi(\Omega,\mathbb{R}^N)$ and $W^{1,\varphi}(\Omega,\mathbb{R}^N)$, respectively. 
They are given by
\[
    g\in L^\varphi(\Omega,\mathbb{R}^N)\ \ 
    \overset{\mathrm{def}}{\Longleftrightarrow}\ \
    g\in L^1(\Omega,\mathbb{R}^N)
    \ \text{and}\ 
    \int_\Omega \varphi(|g|)\, dx < \infty,
\]
equipped with the Luxemburg norm $\|g\|_{L^\varphi(\Omega,\mathbb{R}^N)}
    := 
    \inf\{\lambda>0 : \int_\Omega \varphi(|g|/\lambda)\, dx \le 1\}$.
Similarly,
\[
    g\in W^{1,\varphi}(\Omega,\mathbb{R}^N)\ \ 
    \overset{\mathrm{def}}{\Longleftrightarrow}\ \ 
    g\in W^{1,1}(\Omega,\mathbb{R}^N)
    \ \text{and}\ 
    |g|, |D g|\in L^\varphi(\Omega,\R^1),
\]
with norm
$
    \|g\|_{W^{1,\varphi}(\Omega,\mathbb{R}^N)}
    := 
    \|g\|_{L^\varphi(\Omega,\R^N)}
    +
    \|Dg\|_{L^\varphi(\Omega,\R^{N\times n})}$.
We denote by $W^{1,\varphi}_0(\Omega,\mathbb{R}^N)$ the closure of 
$C_0^\infty(\Omega,\mathbb{R}^N)$ in $W^{1,\varphi}(\Omega,\mathbb{R}^N)$.  
For simplicity, we write 
$\|g\|_{L^\varphi(\Omega)}=\|g\|_{L^\varphi(\Omega,\mathbb{R}^N)}$ and 
$\|g\|_{W^{1,\varphi}(\Omega)}=\|g\|_{W^{1,\varphi}(\Omega,\mathbb{R}^N)}$.  
In the same manner, replacing $\varphi$ by $H(x,\cdot)$ above, we define the Lebesgue space $L^H(\Omega,\mathbb{R}^N)$ and 
the Sobolev spaces $W^{1,H}(\Omega,\mathbb{R}^N)$ and $W^{1,H}_0(\Omega,\mathbb{R}^N)$ 
associated with the double phase function $H$ defined in~\eqref{H}.

We introduce an important concept: the \textit{shifted} N-functions $\varphi_s$ where $s\geq0$; see \cite{DE08}. We define 
\begin{equation}\label{def:shift}
    \varphi_s(t):=\int_0^t\varphi_s'(\tau)\, d\tau \quad\text{with}\quad \varphi_s'(t):=\frac{\varphi'(s+t)}{s+t} t,\qquad t\geq0.
\end{equation}
Note that, as stated in \cite[Proposition~2.3]{CO20}, the functions $\varphi_s$
satisfy the $\Delta_2$-condition uniformly in $s\ge 0$ and we have
\begin{equation*}
    \varphi_s(t)\sim\varphi_s'(t)t;
\end{equation*}
\begin{equation}\label{eq:shifteqv}
    \varphi_s(t)\sim\varphi''(s+t)t^2\sim\frac{\varphi(s+t)}{(s+t)^2}t^2\sim\frac{\varphi'(s+t)}{s+t}t^2;
\end{equation}
\begin{equation}\label{eq:addtoshift}
    \varphi(s+t)\sim[\varphi_s(t)+\varphi(s)].
\end{equation}
Also, by \cite[Lemma 30]{DE08}, 
\begin{equation*}
    \varphi_s^*(\kappa\varphi'(s))\sim\kappa^2\varphi(s)
\end{equation*}
uniformly in $\kappa\in[0,1]$ and $s\geq0$. Note that for any $x\in \Omega$, the N-function $H(x,\cdot)$ satisfies all the properties above about shifting, uniformly in $x$. 
In the case of $H$ in (\ref{H}), denoting $\psi(t):=t^p$ and $ \widetilde \psi(t):=t^q$, we have that for $s,t\geq0$,
\begin{equation}\label{eq:shiftislinear}
    H_s'(x,t)=p(t+s)^{p-2}t+a(x)q(s+t)^{q-2}t =\psi_s'(t)+a(x)\widetilde{\psi}_s'(t).
\end{equation}

We next define
\begin{equation}\label{def:V}
    V_\varphi(P):=\sqrt{\frac{\varphi'(|P|)}{|P|}}P
    \ \ \text{and}\ \      V_H(x,P):=\sqrt{\frac{H'(x,|P|)}{|P|}}P
    ,\quad P\in\mathbb{R}^{N\times n}.
\end{equation}
and in particular, if $\varphi(t)=t^p$, we reduce $V_\varphi$ to $V_p$. Note that
\begin{equation}\label{eq:Vsquare}
|V_\varphi(P)|^2\sim \varphi(|P|).    
\end{equation}
We have from \cite[Lemma 7]{DLSV12} that for every $P,Q\in\mathbb{R}^{N\times n}$,
\begin{equation}\label{eq:Vtoshift}
    |V_\varphi(P)-V_\varphi(Q)|^2\sim\varphi_{|Q|}(|P-Q|)\sim \frac{\varphi'(|P|+|Q|)}{|P|+|Q|}|P-Q|^2.
\end{equation}
In particular, as in \cite[(2.14)]{DLSV12}, we also have from \eqref{C3} that for every $x\in\Omega$ and every $P,Q\in\mathbb{R}^{N\times n}$,
\begin{equation}\label{eq:differencetoshift}
\begin{aligned}
    |\partial f(x,P)-\partial f(x,Q)|&\leq \left(\int_0^1|\partial^2 f(x,P+t(Q-P))|\, dt \right)|P-Q|\\
    &\leq cH''(x,|P|+|P-Q|)|P-Q|\\
    &\leq cH'_{|P|}(x,|P-Q|)
\end{aligned}
\end{equation}
Recall from \cite[Lemma A.2]{DKS12} that for $g\in L^{\varphi}(B_r(x_0),\mathbb{R}^{N\times n})$,
\begin{equation}\label{eq:Vmean}
    \fint_{B_r(x_0)}|V_\varphi(g)-V_\varphi((g)_{x_0,r})|\, dx\sim\fint_{B_r(x_0)}|V_\varphi(g)-(V_\varphi(g))_{x_0,r}|\, dx.
\end{equation}

Recall the modulating coefficient $a:\Omega\to [0,\infty)$ appearing in the double phase function $H$ in \eqref{H}, which is assumed to satisfy $a\in C^{0,\alpha}(\Omega)$. 
Given an open set $U\Subset\Omega$, we define points $x_U^\pm\in\overline{U}$ such that
\begin{equation} \label{eq:pm}
    a_{U}^+:=a(x_{U}^+)=\sup_{x\in U}a(x)\quad\text{ and }\quad a_{U}^-:=a(x_{U}^-)=\inf_{x\in U}a(x).
\end{equation}
In particular, if $U=B_r(x_0)$, we also write $x_{x_0,r}^\pm:=x_{B_r(x_0)}^\pm$ and $a_{x_0,r}^\pm:=a_{B_r(x_0)}^\pm$. In addition, we write
\begin{equation}\label{eq:frozenH}
    H_{U}^\pm(t):=H(x_{U}^\pm,t).
\end{equation}

We introduce Sobolev--Poincar\'e type inequalities for the double phase function $H$. 
From \cite[Theorem~2.13]{O17} (see also \cite[Lemma~2.4]{OSS25}), we first obtain the following estimate.
\begin{lem}\label{lem:sobolevpoincare}
    (Sobolev-Poincar\'e inequality) Let $H$ be defined as in (\ref{H}) and (\ref{a}). Then there exists $\theta_0=\theta_0(n,p,q)\in(0,1)$ such that for any $w\in W^{1,1}(\Omega,\mathbb{R}^N)$ and $B_r=B_r(x_0)\Subset\Omega$ with $r\leq1$, we have
    \begin{equation*}\begin{aligned}
        &\fint_{B_r}H\left(x^+_{x_0,r},\frac{|w-(w)_{x_0,r}|}{r}\right)\, dx\\
        &\leq c(1+[a]_{C^{0,\alpha}}\|Dw\|_{L^p(B_r)}^{q-p})\left(\fint_{B_r}H(x_{x_0,r}^-,|Dw|)^{\theta_0}\, dx\right)^\frac{1}{\theta_0},
    \end{aligned}
    \end{equation*}
    for some $c=c(n,p,q)\geq1$.
\end{lem}

Moreover, a Sobolev--Poincar\'e inequality for the shifted function $H_{|Q|}$ under an a priori higher integrability assumption is obtained in \cite[Lemma~2.4]{OSS25}.

\begin{lem}\label{lem:shiftsobolevpoincare}
    Let $H$ be defined as in (\ref{H}) and (\ref{a}), $Q\in\mathbb{R}^{N\times n}\setminus\{0\}$. Then there exists $\theta=\theta(n,p,q)\in (0,1)$ such that for any $w\in W^{1,1}(\Omega,\R^N)$ with $Dw\in L_{\mathrm{loc}}^{p(1+\sigma_0)}(\Omega,\R^{N\times n})$ for some $\sigma_0>0$, and $B_r(x_0)\Subset\Omega$ with $r\leq1$ satisfying $\|Dw\|_{L^{p(1+\sigma_0)}(\Omega)}\leq1$, we have
    \begin{equation*}
    \begin{aligned}
        \fint_{B_r(x_0)}H_{|Q|}\left(x_{x_0,r}^+,\frac{|w-(w)_{x_0,r}|}{r}\right)\, dx&\leq c\left(\fint_{B_r(x_0)}H_{|Q|}(x_{x_0,r}^-,|Dw|))^\theta\, dx\right)^\frac{1}{\theta}\\
        &\quad+c(r^{\alpha(\sigma_0)}+r^\alpha|Q|^{q-p})|Q|^p,
    \end{aligned}
    \end{equation*}
    for some $c=c(n,p,q,\alpha,[a]_{C^{0,\alpha}})\geq1$, where $\alpha(\sigma_0):=\alpha-\frac{(q-p)n}{p(1+\sigma_0)}>0$.
\end{lem}

{The following lemma about the almost concave condition is from \cite[Lemma 2.2]{O18}.  
\begin{lem}
    Let $\Psi : [0,\infty)\to[0,\infty)$ be nondecreasing and such that $t\mapsto\frac{\Psi(t)}{t}$ is nonincreasing. Then there exists a concave function $\Bar{\Psi}\ :\ [0,\infty)\to[0,\infty)$ such that
    \begin{equation*}
        \frac{1}{2}\Bar{\Psi}(t)\leq\Psi(t)\leq\Bar{\Psi}(t)\qquad\text{ for all }\ t\geq0.
    \end{equation*}
\end{lem}
}
\subsection{Approximation and regularity results for autonomous problems}
We begin with the version of $\mathcal{A}$-harmonic approximation from \cite[Lemma 2.7]{CO20}, which originally refers to \cite[Theorem 14]{DLSV12}.
\begin{lem}\label{Aharmonicapproximation}
    Let $\mathcal{A}$ be a bilinear form on $\mathbb{R}^{N\times n}$ satisfying the Legendre-Hadamard condition, that is, there are $0< \nu_\mathcal{A} \le L_\mathcal{A}$ such that for every $b\in\mathbb{R}^N$ and  $z\in\mathbb{R}^n$,
    \begin{equation}\label{eq:LH1}
        \nu_\mathcal{A}|b|^2|z|^2\leq\langle\, \mathcal{A}(b\otimes z)\, |\, (b\otimes z)\, \rangle\leq L_\mathcal{A}|b|^2|z|^2,
    \end{equation}
    and let $\varphi$ be an N-function with $\varphi,\varphi^*$ satisfying $\Delta_2$-condition, and let $\sigma,\mu>0$. Then for every $\varepsilon>0$, there exists $\delta=\delta(n,N,\nu_\mathcal{A},L_\mathcal{A},\Delta_2(\varphi,\varphi^*),\sigma,\varepsilon)>0$ such that whenever $w\in W^{1,\varphi}(B_r,\R^N)$ satisfies
    \begin{equation*}
        \fint_{B_r}\varphi(|Dw|)\, dx\leq\left(\fint_{B_r}\varphi(|Dw|)^{1+\sigma}\, dx\right)^\frac{1}{1+\sigma}\leq\varphi(\mu),
    \end{equation*}
    and is almost $\mathcal{A}$-harmonic in $B_r$ in the sense that
    \begin{equation}\label{eq:almostAsense}
        \left|\fint_{B_r}\langle\, \mathcal{A}Dw\, |\, D\zeta\, \rangle\, dx\right|\leq\delta\mu\|D\zeta\|_{L^\infty(B_r)}
    \end{equation}
  for all $\zeta\in W_0^{1,\infty}(B_r,\R^{N})$, we have
    \begin{equation}\label{eq:Aapproximatesense}
        \fint_{B_r}\varphi\left(\frac{|w-h|}{r}\right)\, dx+\fint_{B_r}\varphi(|Dw-Dh|)\, dx\leq\varepsilon\varphi(\mu),
    \end{equation}
    where $h\in W^{1,\varphi}(B_r,\R^N)$ is the unique weak solution of the problem
    \begin{equation*}
        \begin{dcases}
            -\div(\mathcal{A}Dh)=0&\text{in }\ B_r,\\
            h =w&\text{on }\ \partial B_r.
        \end{dcases}
    \end{equation*}
\end{lem}

The next result is an approximation result for almost minimizers for the $\phi$-energy functional $v\mapsto \int_{\Omega}\phi(|Dv|)\,dx$. In \cite[Lemma 4.13]{HO23}, the authors establish an approximation lemma for almost minimizers of energy functionals with general $\phi$‑growth and 
$\phi$‑ellipticity in the scalar case. The same argument applies to the $\phi$‑energy functional in the vectorial case. Therefore, we state the following approximation result without providing a proof.

\begin{lem}[$\phi$-minimizing approximation]\label{lem:minimizingapproximation}
    Let $\phi\in C^1([0,\infty))\cap C^2((0,\infty))$ satisfy  \eqref{eq:N'incrange}, $\delta\in(0,1)$, and $\mu>0$.  Suppose that $w\in W^{1,\varphi}(B_r,\R^N)$ satisfies  that  for some $\sigma,\gamma,\tilde c>0$,
    \begin{equation*}
        \fint_{B_r}\phi(|Dw|)^{1+\sigma}\, dx\leq \bar{c}\mu^{1+\sigma}
    \end{equation*}
    and
    \begin{equation*}
        \fint_{B_{r}} \phi(|Dw|)\, dx\leq\fint_{B_{r}}\phi(|Dw+D\zeta|)\, dx+\tilde c\delta\left(\frac{\|D\zeta\|_{L^\infty(B_r)}}{\phi^{-1}(\mu)}+1\right)^\gamma \mu
    \end{equation*}
 for all $\zeta\in W_0^{1,\infty}(B_r)$.
    Then there exists $c=c(n,N,p_1,p_2,\sigma,\gamma,\bar{c},\tilde c)>0$ such that
    \begin{equation*}
        \fint_{B_{r}}\frac{\phi'(|Dw|+|Dh|)}{|Dw|+|Dh|}|Dw-Dh|^2\, dx\leq c\delta^{\Bar{\gamma}}\mu,
    \end{equation*}
    where $\Bar{\gamma}:=\frac{\sigma p_1}{\gamma+\sigma p_1}$ and $h \in w+ W_0^{1,\phi}(B_r(x_0),\R^N)$ is the minimizer of the energy functional  
 $$
 w+ W_0^{1,\phi}(B_r,\R^N)\ni v \ \ \mapsto\ \  \int_{B_r} \phi(|Dv|)\, dx.
 $$
\end{lem}

The following lemma provides the supremum estimate and the excess decay estimate for the gradient of  $\overline{\phi}$-harmonic mappings, where
    \begin{equation}\label{eq:phiD}
        \overline{\phi}(t):=t^p+at^q\quad\text{ with }\quad1<p\leq q\quad\text{ and }\quad a\geq0.
        \end{equation}
For these estimates, we refer to \cite[Lemma 5.8 and Theorem 6.4]{DSV09} in the case $\phi=\overline{\phi}$.
\begin{lem}\label{lem:gamma0}
    Let    $\overline{\phi}:[0,\infty)\to [0,\infty)$ be as above in \eqref{eq:phiD}. 
    There exist constants $c>0$ and $\gamma_0\in(0,1)$ depending on $n$, $N$, $p$ and $q$ and  independent of $a\ge 0$, such that if $\bar{h}\in W^{1,\overline{\phi}}(B_R(x_0),\R^N)$ is a weak solution to the $\overline\phi$-Laplace system
    \begin{equation*}
        \div\left(\frac{\overline{\phi}'(|D\bar{h}|)}{|D\bar{h}|}D\bar{h}\right)=0\quad\text{ in }\quad B_R(x_0),
    \end{equation*}
    then  we have that for every $\tau\in(0,1]$,
    \begin{equation*}
        \sup_{B_{\frac{\tau R}{2}}(x_0)}\overline{\phi}(|D\bar{h}|)\leq c\fint_{B_{\tau R}(x_0)}\overline{\phi}(|D\bar{h}|)\, dx,
    \end{equation*}
    and
    \begin{equation*}
        \fint_{B_{\tau R}(x_0)}|V_{\overline{\phi}}(D\bar{h})-(V_{\overline{\phi}}(D\bar{h}))_{x_0,\tau R}|^2\, dx\leq c\tau^{2\gamma_0}\fint_{B_{R}(x_0)}|V_{\overline{\phi}}(D\bar{h})-(V_{\overline{\phi}}(D\bar{h}))_{x_0, R}|^2\, dx.
    \end{equation*}
\end{lem}

\section{Caccioppoli and Self-improving estimates}

We derive various self-improving estimates, also known as higher integrability estimates, for the gradients of minimizers of $\mathcal F$ in \eqref{F}. To this end, we obtain Caccioppoli type estimates for the minimizers.  

We will use the following two lemmas. The first is a technical iteration lemma from \cite[Chapter~V, Lemma~3.1]{Gia83}, and the second is a variant of the results by Gehring \cite{Gh73} and by Giaquinta and Modica \cite[Theorem~6.6]{Giusti}.

\begin{lem}\label{lem:it}
    Let $g:[r,R]\subset\mathbb{R}^+\to [0,\infty)$ be a bounded function. Assume further that for all $r\leq\rho_1<\rho_2 \leq R$ we have
    \begin{equation*}
        g(\rho_1)\leq \theta g(\rho_2) + C_1(\rho_2-\rho_1)^{-\gamma_1}+C_2
    \end{equation*}
    for some nonnegative constants $C_1$ and $C_2$, nonnegative exponent $\gamma_1$, and a parameter $\theta\in[0,1)$. Then we have
    \begin{equation*}
        g(r)\le  c \big( C_1(R-r)^{-\gamma_1}+C_2\big)
    \end{equation*}
    for some $c=c(\theta,C_1,C_2)>0$.
\end{lem}

\begin{lem}\label{lem:gehring}
    Let $B_0\subset\mathbb{R}^n$ be a ball, $g\in L^1(B_0)$ and $h\in L^{s_0}(B_0)$ for some $s_0>1$. Assume that for some $\gamma\in(0,1)$, $c_1>0$ and all balls $B$ with $2B\subset B_0$, where $2B$ denotes the ball centered at the same point as $B$ with radius twice that of $B$,
    \begin{equation*}
        \fint_B|g|\, dx\leq c_1\left(\fint_{2B}|g|^\gamma\, dx\right)^\frac{1}{\gamma}+c_1\fint_{2B}|h|\, dx.
    \end{equation*}
    Then there exist $s_1,c_2>0$ depending on $s_0$, $\gamma$, $c_1$ and $n$ such that $h\in L_{\mathrm{loc}}^{s_1}(B)$ and for all $s_2\in[1,s_1]$,
    \begin{equation*}
        \left(\fint_{B}|g|^{s_2}\, dx\right)^\frac{1}{s_2}\leq c_2\fint_{2B}|g|\, dx+c_2\left(\fint_{2B}|h|^{s_2}\, dx\right)^\frac{1}{s_2}.
    \end{equation*}
\end{lem}

\subsection{Self-improving estimates for gradients}
We begin with a basic Caccioppoli-type estimate for the minimizers of (\ref{F}). Note that in this subsection, we may assume that the function $f(x,P)$ in \eqref{F} satisfies only \eqref{C2} and the first inequality in \eqref{ass:gc}.

\begin{lem}\label{lem:Caccioppoli1} 
Let $u\in W^{1,p}(\Omega,\mathbb{R}^N)$ with $H(\cdot,|Du|)\in L^1(\Omega)$ be a  minimizer of (\ref{F}). Then there exists $c=c(n,N,p,q,L,\nu)>0$ such that for all balls $B_{2r}(x_0)\Subset\Omega$, 
    \begin{equation*}
        \int_{B_r(x_0)}H(x,|Du|)\,dx\leq c\int_{B_{2r}(x_0)}H\left(x,\frac{|u-(u)_{x_0,2r}|}{r}\right)\,dx.
    \end{equation*}
\end{lem}
\begin{proof}
Throughout the proof, we omit the center $x_0$ from the notation. For $0<r\leq r_1<r_2\leq2r$ and $\eta\in C_c^\infty(B_{r_2})$ with $0\leq\eta\leq1$, $\eta=1$ in $B_{r_1}$ and $|D\eta|\leq\frac{c}{r_2-r_1}$, define $\zeta:=\eta(u-(u)_{x_0,2r})$. Then using \eqref{C2} and the first inequality in \eqref{ass:gc},
\begin{equation*}
\begin{aligned}
    \int_{B_{r_1}} H(x,|Du|)\, dx &\leq c\int_{B_{r_1}}f(x,Du)-f(x,\mathbf{0})\, dx\\
    &\leq c\int_{B_{r_2}}f(x,Du)-f(x,\mathbf{0})\, dx\\
    &\leq c\int_{B_{r_2}}f(x,Du-D\zeta)-f(x,\mathbf{0})\, dx\\
    &\leq c\int_{B_{r_2}}H(x,|Du-D\zeta|)\, dx\\
    &\leq c^*\int_{B_{r_2}\setminus B_{r_1}}H(x,|Du|)\, dx+c\int_{B_{r_2}}H\left(x,\frac{|u-(u)_{2r}|}{r_2-r_1}\right)\, dx.
\end{aligned}
\end{equation*}
Filling the hole by adding $c^*\int_{ B_{r_1}}H(x,|Du|)\, dx$ and then writing $r_1=sr$ and $r_2=tr$ for some $1\leq s<t\leq2$, we obtain
\begin{equation*}
    \begin{aligned}
        \int_{B_{sr}} H(x,|Du|)\, dx&\leq \frac{c^*}{1+c^*}\int_{B_{tr}}H(x,|Du|)\, dx+\frac{c}{(t-s)^q}\int_{B_{tr}}H\left(x,\frac{|u-(u)_{2r}|}{r}\right)\, dx.
    \end{aligned}
\end{equation*}
Finally, using the iteration of Lemma \ref{lem:it}, we have the assertion.
\end{proof}

 Applying Lemmas \ref{lem:sobolevpoincare} and  \ref{lem:gehring} to Lemma \ref{lem:Caccioppoli1}, we  obtain the higher integrability result as follows.
\begin{lem}\label{lem:Highertintegrability0} 
    Let $H$ satisfy \eqref{a}, and  $u\in W^{1,p}(\Omega,\mathbb{R}^N)$ with $H(\cdot,|Du|)\in L^1(\Omega)$ be a minimizer of (\ref{F}). Then there exist constants $\sigma_0>0$ and $c>0$ depending on $n,N,p,q,\alpha,[a]_{C^{0,\alpha}},L$ and $\nu$ such that for any $B_{2r}(x_0)\Subset\Omega$ with $\int_{B_{2r}(x_0)}H(x,|Du|)\, dx\leq1$, we have $H(\cdot,|Du|)\in L_{\mathrm{loc}}^{1+\sigma_0}(\Omega)$ with
    \begin{equation*}
        \left(\fint_{B_r(x_0)}H(x,|Du|)^{1+\sigma_0}\, dx\right)^{\frac{1}{1+\sigma_0}}\leq c\fint_{B_{2r}(x_0)}H(x,|Du|)\, dx.
    \end{equation*}
    Moreover, in view of \cite[Remark 6.12]{Giusti}, for each $t\in(0,1)$, there exists a constant $c_t>0$ depending also on $t$ such that
     \begin{equation*}
        \left(\fint_{B_r(x_0)}H(x,|Du|)^{1+\sigma_0}\, dx\right)^{\frac{1}{1+\sigma_0}}\leq c_t\left(\fint_{B_{2r}(x_0)}H(x,|Du|)^t\, dx\right)^\frac{1}{t}.
    \end{equation*}
\end{lem}
Here, we remark that the following smallness condition arises from the previous lemma.
\begin{rem}
    Since $H(\cdot,|Du|)\in L_{\mathrm{loc}}^{1+\sigma_0}(\Omega)$, given $\Omega'\Subset\Omega$, there exists $r_0\in(0,1]$ such that for any $B_{2r}(x_0)\subset\Omega'$ with $r\in(0,r_0]$,
    \begin{equation}\label{eq:smallness}
        |B_{2r}(x_0)|\leq1\qquad\text{and}\qquad\int_{B_{2r}(x_0)}H(x,|Du|)^{1+\sigma_0}\, dx\leq1.
    \end{equation}
    Note that (\ref{eq:smallness}) implies 
    \begin{equation*}
    \int_{B_{2r}}H(x,|Du|)\, dx\leq\frac{1}{1+\sigma_0}\int_{B_{2r}}H(x,|Du|)^{1+\sigma_0}\, dx+\frac{\sigma_0}{1+\sigma_0}|B_{2r}|\leq1.
\end{equation*}
\end{rem}

Combining the above higher integrability result with the argument in \cite[Lemma 3.3]{OSS25}, we further obtain the following estimates.

\begin{lem}
   Let  $H$ satisfy \eqref{a}, and   $u\in W^{1,p}(\Omega,\mathbb{R}^N)$ with $H(\cdot,|Du|)\in L^1(\Omega)$ be a minimizer of (\ref{F}). For any $B_{2r}(x_0)\Subset\Omega$ with $\int_{B_{2r}}H(x,|Du|)\, dx\leq1$ and $r\leq\frac{1}{2}$, there exists $c=c(n,N,p,q,\alpha,[a]_{C^{0,\alpha}},L,\nu)>0$ such that we have
    \begin{equation}\label{eq:Hoderintegrability}
        \left(\fint_{B_r(x_0)}H(x,|Du|)^{1+\sigma_0}\,dx\right)^{\frac{1}{1+\sigma_0}}\leq cH_{B_{2r}(x_0)}^-\left(\fint_{B_{2r}(x_0)}|Du|\,dx\right),
    \end{equation}
    and
    \begin{equation}\label{eq:H'oderintegrability}
        \fint_{B_r(x_0)}(H_{B_{2r}(x_0)}^-)'(|Du|)\,dx\leq c(H_{B_{2r}(x_0)}^-)'\left(\fint_{B_{2r}(x_0)}|Du|\,dx\right).
    \end{equation}
\end{lem}

We also recall two comparison estimates from \cite[Lemma 3.6]{OSS25} that compare the functions $H(x,t)$ and $H_{B_{2r}}^-(t)$, and $H'(x,t)$ and $(H_{B_{2r}}^-)'(t)$. Note here that $H(x,t)-H_{B_{2r}}^-(t)=(a(x)-a_{B_{2r}}^-)t^q$ and $H'(x,t)-(H_{B_{2r}}^-)'(t)=q(a(x)-a_{B_{2r}}^-)t^{q-1}$.

\begin{lem}\label{lem:qintegrability} 
  Let $H$ satisfy \eqref{a}, and   $u\in W^{1,p}(\Omega,\mathbb{R}^N)$ with $H(\cdot,|Du|)\in L^1(\Omega)$ be a minimizer of (\ref{F}). There exist $\alpha_0=\alpha_0(n,N,p,q,\alpha,[a]_{C^{0,\alpha}},L,\nu)>0$ and \\
  $c=c(n,N,p,q,\alpha,[a]_{C^{0,\alpha}},L,\nu)>0$ such that for every $B_{2r}(x_0)\Subset\Omega$ with $r\leq\frac{1}{2}$ satisfying (\ref{eq:smallness}),
\begin{equation}\label{eq:qintegrabilityq}
    \fint_{B_r(x_0)}(a(x)-a_{x_0,2r}^-)|Du|^q\, dx \leq cr^{\alpha_0}H^-_{B_{2r}(x_0)}\left(\fint_{B_{2r}(x_0)}|Du|\, dx\right)
\end{equation}
and
\begin{equation}\label{eq:qintegrabilityq-1}
    \fint_{B_r(x_0)}(a(x)-a_{x_0,2r}^-)|Du|^{q-1}\, dx\leq cr^{\alpha_0}(H^-_{B_{2r}(x_0)})'\left(\fint_{B_{2r}(x_0)}|Du|\, dx\right).
\end{equation}

\end{lem}

\subsection{Self-improving estimates in the shifted setting}
We begin by deriving a Caccioppoli-type estimate in the shifted setting. Note that in this subsection, we assume that the function $H$ in \eqref{H} satisfies \eqref{a} and the function $f(x,P)$ in \eqref{F} satisfies \eqref{C1}--\eqref{C5}.

\begin{lem}\label{lem:Caccioppoli2}
Let   $u\in W^{1,p}(\Omega,\mathbb{R}^N)$ with $H(\cdot,|Du|)\in L^1(\Omega)$ be a minimizer of (\ref{F}).   For every $B_{2r}(x_0)\Subset\Omega$ with $r\leq\frac{1}{2}$ satisfying (\ref{eq:smallness}) and for every affine function $\ell :\mathbb{R}^n\rightarrow\mathbb{R}^N$ defined by 
\begin{equation*}
    \ell(x):=Q(x-x_0)+(u)_{x_0,2r},\qquad x\in\mathbb{R}^n,
\end{equation*}
with $Q\in\mathbb{R}^{N\times n}$ and $b\in \mathbb R^N$, there exists $c=c(n,N,p,q,\alpha,[a]_{C^{0,\alpha}},L,\nu)>0$ such that
\begin{equation*}
\begin{aligned}
    &\fint_{B_r(x_0)}H_{|Q|}(x,|Du-Q|)\, dx\\
    &\leq c\fint_{B_{2r}(x_0)} H_{|Q|}\left(x,\frac{|u-\ell |}{r}\right)\, dx+c(r^{\beta_1}+r^{\alpha}|Q|^{q-p})H^-_{B_{2r}(x_0)}(|Q|),
\end{aligned}
\end{equation*}
where 
\begin{equation}\label{alpha1}
\alpha_1:=\min\{\alpha,\alpha_0,\beta_0\}.
\end{equation}
Here, $\alpha$, $\alpha_0$ and $\beta_0$ are given in \eqref{a}, Lemma~\ref{lem:qintegrability}, and \eqref{C5}, respectively.
\end{lem}
\begin{proof}
Throughout the proof, we omit the dependence on the center $x_0$ from the notation and write $x^\pm:=x_{x_0,2r}^\pm$ (see \eqref{eq:pm}). Let $r\leq r_1<r_2<\frac{4}{3}r$. For a cut-off function $\eta\in C_0^\infty(B_{r_2})$ satisfying $0\leq\eta\leq1$, $\eta=1$ in $B_{r_1}$ and $|D\eta|\leq\frac{c}{r_2-r_1}$, we define $\zeta:=\eta(u-\ell )$ and $\xi:=(1-\eta)(u-\ell )$ on $\Omega$. Note that $\zeta+\xi=u-\ell $ and $D\zeta+D\xi =Du-Q$.

Using the quasiconvexity \eqref{C4}, we have
\begin{equation*}
    \begin{aligned}
        \int_{B_{r_2}}H_{|Q|}(x^-,|D\zeta|)dx&\leq\int_{B_{r_2}}f(x^-,Q+D\zeta)-f(x^-,Q)\, dx\\
        &=\underbrace{\int_{B_{r_2}}f(x^-,Q+D\zeta)-f(x^-,Q+D\zeta+D\xi )\, dx}_{=:I_1}\\
        &\qquad+\underbrace{\int_{B_{r_2}}f(x^-,Du)-f(x,Du)\, dx}_{=:I_2}\\
        &\qquad+\underbrace{\int_{B_{r_2}}f(x,Du)-f(x,Du-D\zeta)\, dx}_{=:I_3}\\
        &\qquad+\underbrace{\int_{B_{r_2}}f(x,Q+D\xi )-f(x^-,Q+D\xi )\, dx}_{=:I_4}\\
        &\qquad+\underbrace{\int_{B_{r_2}}f(x^-,Q+D\xi )-f(x^-,Q)\, dx}_{=:I_5}.
    \end{aligned}
    \end{equation*}
    Note that, since $u$ minimizes $\mathcal{F}$, $I_3\leq0$.
    
     $I_1+I_5$ can be estimated by (\ref{eq:differencetoshift}) as follows,
        \begin{equation*}
            \begin{aligned}
                I_1+I_5&= \int_{B_{r_2}}\int_0^1\langle\,  \partial f(x^-,Q+tD\xi )-\partial f(x^-,Q+D\zeta+tD\xi )\, |\, D\xi \, \rangle \, dt\, dx\\
                &=\int_{B_{r_2}}\int_0^1\langle\, \partial f(x^-,Q+tD\xi )-\partial f(x^-,Q)\\
                &\qquad\qquad\qquad+ \partial f(x^-,Q)-\partial f(x^-,Q+D\zeta+tD\xi )\, |\, D\xi \, \rangle \, dt\, dx\\
                &\leq c\int_{B_{r_2}}H_{|Q|}(x^-,|D\xi |)\, dx+c\int_{B_{r_2}}H'_{|Q|}(x^-,|D\zeta|+|D\xi |)|D\xi |\, dx.
            \end{aligned}
        \end{equation*}
        For every $\delta\in(0,1)$, by (\ref{eq:Nyoung}) and (\ref{eq:conjugatecompositedifferentiation}), the shifted N-functions $(H_{B_r}^-)_s(t)=H_s(x^-,t)$ with $s\ge 0$ satisfy 
        $\tau H'_{s}(x^-,t)\leq c(\delta)H_{s}(x^-,\tau)+\delta H_{s}(x^-,t)$, so that we have
        \begin{equation*}
        \begin{aligned}
            H'_{|Q|}(x^-,|D\zeta|+|D\xi |)|D\xi |&\leq \delta H_{|Q|}(x^-,|D\zeta|+|D\xi |)+c(\delta)H_{|Q|}(x^-,|D\xi |)\\
            &\leq\delta H_{|Q|}(x^-,|D\zeta|)+c(\delta)H_{|Q|}(x^-,|D\xi |).
        \end{aligned}
        \end{equation*}
        Thus we have
        \begin{equation*}
            I_1+I_5\leq c\int_{B_{r_2}}H_{|Q|}(x^-,|D\xi |)\, dx+\frac{1}{2}\int_{B_{r_2}}H_{|Q|}(x^-,|D\zeta|)\, dx.
        \end{equation*}
      By \eqref{C5}, we estimate $I_2+I_4$ as
        \begin{equation}\label{eq:I24_1}
            \begin{aligned}
                I_2+I_4&\leq c\int_{B_{r_2}}r^{\beta_0}H(x,|Du|)+(a(x)-a(x^-))|Du|^q \, dx\\
                &\quad+c\int_{B_{r_2}}r^{\beta_0}H(x,|Q+D\xi |)+(a(x)-a(x^-))|Q+D\xi |^q \, dx.
            \end{aligned}
        \end{equation}
    Applying (\ref{eq:shiftislinear}) and (\ref{eq:addtoshift}) to the N-function $\widetilde\psi(t):=t^q$,   we find
    \begin{equation}\label{eq:I24_2}
    \begin{aligned}
        &\int_{B_{r_2}}(a(x)-a(x^-))|Q+D\xi |^q \, dx 
        \leq c\int_{B_{r_2}}r^\alpha|Q|^q+(a(x)-a(x^-))\widetilde\psi(|D\xi |)\, dx\\
        &\leq c\int_{B_{r_2}}r^\alpha|Q|^{q-p}H(x^-,|Q|) + (a(x)-a(x^-))\left(\widetilde\psi_{|Q|}(|D\xi |)+|Q|^q\right)\, dx\\
        &\leq c\int_{B_{r_2}}r^\alpha|Q|^{q-p}H(x^-,|Q|)+H_{|Q|}(x,|D\xi |)\, dx.
    \end{aligned}
    \end{equation}
    Furthermore, using (\ref{eq:qintegrabilityq}) and Jensen's inequality, we see that
    \begin{equation}\label{eq:Duqintegrable}
    \begin{aligned}
        \int_{B_{\frac{4}{3}r}}(a(x)-a(x^-))|Du|^q\, dx&\leq r^{\alpha_0+n}H\bigg(x^-,\fint_{B_{\frac{5}{3}r}}|Du|\, dx\bigg)\\
        &\leq r^{\alpha_0}\int_{B_{\frac{5}{3}r}}H(x,|Du|)\, dx.
    \end{aligned}
    \end{equation}  
    Therefore, collecting (\ref{eq:I24_1}), (\ref{eq:I24_2}) and (\ref{eq:Duqintegrable}), and using (\ref{eq:addtoshift}), we obtain
    \begin{equation*}
    \begin{aligned}
        I_2+I_4&\leq c\int_{B_{\frac{5}{3}r}}r^{\beta_0}H(x,|Du|)+r^{\alpha_0}H(x,|Du|)\, dx\\
        &\qquad+c\int_{B_{r_2}} r^{\beta_0}H(x,|Q+D\xi |)+H_{|Q|}(x,|D\xi |)\, dx+cr^{\alpha+n}|Q|^{q-p}H(x^-,|Q|)\\
        &\leq c\int_{B_{\frac{5}{3}r}}r^{\beta_0}H(x,|Du|)+r^{\alpha_0}H(x,|Du|)\, dx\\
        &\qquad+c\int_{B_{r_2}} r^{\beta_0}\big(H(x^-,|Q|)+r^\alpha|Q|^q\big)+H_{|Q|}(x,|D\xi |)\, dx\\
        &\qquad+cr^{\alpha+n}|Q|^{q-p}H(x^-,|Q|)\\
        &\leq c\int_{B_{\frac{5}{3}r}}r^{\beta_1}H(x,|Du|)\, dx+c\int_{B_{r_2}} H_{|Q|}(x,|D\xi |)\, dx\\
        &\qquad+cr^n(r^{\beta_0}+r^{\alpha})|Q|^{q-p}H(x^-,|Q|).
    \end{aligned}
    \end{equation*}
    
    We then combine the above estimates and use (\ref{eq:addtoshift}) to find
    \begin{equation*}
    \begin{aligned}        
        \int_{B_{r_2}}H_{|Q|}(x^-,|D\zeta|)\, dx&\leq c\int_{B_{r_2}}H_{|Q|}(x,|D\xi |)\, dx+\frac{1}{2}\int_{B_{r_2}}H_{|Q|}(x^-,|D\zeta|)\, dx\\
        &\qquad+cr^{\beta_1}\int_{B_{\frac{5}{3}r}}H(x,|Du|)\, dx+cr^n(r^{\beta_0}+r^{\alpha}|Q|^{q-p})H(x^-,|Q|).
    \end{aligned}
    \end{equation*}
    Subtracting $\frac{1}{2}\int_{B_{r_1}}H_{|Q|}(x^-,|D\zeta|)\, dx$ from both sides, we have that
    \begin{equation*}
        \begin{aligned}
            \int_{B_{r_2}}H_{|Q|}(x^-,|D\zeta|)\, dx&\leq c\int_{B_{r_2}}H_{|Q|}(x,|D\xi |)\, dx+cr^{\beta_1}\int_{B_{\frac{5}{3}r}}H(x,|Du|)\, dx\\
            &\qquad+cr^n(r^{\beta_0}+r^{\alpha}|Q|^{q-p})H(x^-,|Q|).
        \end{aligned}
    \end{equation*}
    We then use the definitions of $\zeta$ and $\xi$, and (\ref{eq:Duqintegrable}) to obtain the following inequality:
    \begin{equation*}
        \begin{aligned}
            &\int_{B_{r_1}}H_{|Q|}(x,|Du-Q|)\, dx \leq\int_{B_{r_2}}H_{|Q|}(x,|D\zeta|)\, dx\\
            &\leq \int_{B_{r_2}}H_{|Q|}(x^-,|D\zeta|)\, dx + \int_{B_{r_2}}(a(x)-a(x^-))(|Du|^q+|Q|^q)\, dx\\
            &\leq c\int_{B_{r_2}}H_{|Q|}(x,|D\xi |)\, dx+cr^{\beta_1}\int_{B_{\frac{5}{3}r}}H(x,|Du|)\, dx+cr^n(r^{\beta_0}+r^{\alpha}|Q|^{q-p})H(x^-,|Q|)\\
            &\leq c\int_{B_{r_2}\setminus B_{r_1}}H_{|Q|}(x,|Du-Q|)+H_{|Q|}\left(x,\frac{|u-\ell |}{r_2-r_1}\right)\, dx\\
            &\qquad +cr^{\beta_1}\int_{B_{\frac{5}{3}r}}H(x,|Du|)\, dx+cr^n(r^{\beta_0}+r^{\alpha}|Q|^{q-p})H(x^-,|Q|).
        \end{aligned}
    \end{equation*}
    Now, let us set $r_1=sr$ and $r_2=tr$ for any $1\leq s<t\leq\frac{4}{3}$ to find
    \begin{equation*}
        \begin{aligned}
            &\int_{B_{sr}}H_{|Q|}(x,|Du-Q|)\, dx\\
            &\leq c^*\int_{B_{tr}\setminus B_{sr}}H_{|Q|}(x,|Du-Q|)\, dx
            +\frac{c}{(t-s)^q}\int_{B_{\frac{5}{3}r}}H_{|Q|}\left(x,\frac{|u-\ell |}{r}\right)\, dx\\
            &\qquad +cr^{\beta_1}\int_{B_{\frac{5}{3}r}}H(x,|Du|)\, dx+cr^n(r^{\beta_0}+r^{\alpha}|Q|^{q-p})H(x^-,|Q|).
        \end{aligned}
    \end{equation*}
 Then adding $c^*\int_{ B_{sr}}H_{|Q|}(x,|Du-Q|)\, dx$ and applying the iteration Lemma \ref{lem:it} yields that
    \begin{equation}\label{eq:Caccioppolisemifinal}
    \begin{aligned}
        \int_{B_{r}}H_{|Q|}(x,|Du-Q|)\, dx&\leq c\int_{B_{\frac{5}{3}r}}H_{|Q|}\left(x,\frac{|u-\ell |}{r}\right)\, dx +cr^{\beta_1}\int_{B_{\frac{5}{3}r}}H(x,|Du|)\, dx\\
        	&\qquad +cr^n(r^{\beta_0}+r^{\alpha}|Q|^{q-p})H(x^-,|Q|).
    \end{aligned}
    \end{equation}
    Finally, using Lemma \ref{lem:Caccioppoli1} and (\ref{eq:addtoshift}),
    \begin{equation*}
    \begin{aligned}
        \int_{B_{\frac{5}{3}r}}H(x,|Du|)\, dx&\leq c\int_{B_{2r}}H\left(x,\frac{|u-Q(x-x_0)-(u)_{2r}|}{r}+2|Q|\right)\, dx\\
        &\leq c\int_{B_{2r}}H_{|Q|}\left(x,\frac{|u-\ell |}{r}\right)\, dx+cr^n\big(H(x^-,|Q|)+r^\alpha|Q|^q\big),
    \end{aligned}
    \end{equation*}
    and hence we plug the estimate into (\ref{eq:Caccioppolisemifinal}): 
    \begin{equation*}
        \fint_{B_{r}}H_{|Q|}(x,|Du-Q|)\, dx\leq c\fint_{B_{2r}}H_{|Q|}\left(x,\frac{|u-\ell |}{r}\right)\, dx+c(r^{\alpha_1}+r^\alpha|Q|^{q-p}) H(x^-,|Q|).
    \end{equation*}
    
\end{proof}

By means of Lemma \ref{lem:shiftsobolevpoincare}, Lemma \ref{lem:gehring} and  Lemma \ref{lem:Caccioppoli2}, we obtain the following self-improving estimates in the shifted setting.

\begin{lem}\label{lem:Higherintegrability2}
    Let $u\in W^{1,p}(\Omega,\mathbb{R}^N)$ with $H(\cdot,|Du|)\in L^1(\Omega)$ be a minimizer of (\ref{F}). Then there exist constants $\sigma_1>0$ and $c>0$ depending on $n,N,p,q,\alpha,[a]_{C^{0,\alpha}},L$ and $\nu$ such that for any $B_{2r}(x_0)\Subset\Omega$ with (\ref{eq:smallness}) and $Q\in\mathbb{R}^{N\times n}$,
    \begin{equation*}
    \begin{aligned}
        \left(\fint_{B_{r}(x_0)}H_{|Q|}(x,|Du-Q|)^{1+\sigma_1}\, dx\right)^{\frac{1}{1+\sigma_1}}
        &\leq c\fint_{B_{2r}(x_0)}H_{|Q|}(x,|Du-Q|)\, dx\\
        &\qquad +c(r^{\alpha_1}+r^\alpha|Q|^{q-p})H_{B_{2r}(x_0)}^+(|Q|),
    \end{aligned}
    \end{equation*}
    where 
    \begin{equation}\label{alpha2}
        \alpha_2:=\min\left\{\alpha_1,\alpha-\frac{(q-p)n}{p(1+\sigma_0)}\right\}=\min\left\{\alpha_0,\beta_0,\alpha-\frac{(q-p)n}{p(1+\sigma_0)}\right\}
    \end{equation}
and $\alpha_1$ is from \eqref{alpha1}.
 Moreover, for every $t\in(0,1]$, there exists a $c_t>0$ depending also on $t$, such that
    \begin{equation}\label{eq:tHigherintegrability2}
    \begin{aligned}
        \left(\fint_{B_{r}(x_0)}H_{|Q|}(x,|Du-Q|)^{1+\sigma_1}\, dx\right)^{\frac{1}{1+\sigma_1}}
        &\leq c_t\left(\fint_{B_{2r}(x_0)}H_{|Q|}(x,|Du-Q|)^t\, dx\right)^\frac{1}{t}\\
        &\qquad +c_t (r^{\alpha_1} +r^\alpha|Q|^{q-p})H_{B_{2r}(x_0)}^+(|Q|).
    \end{aligned}
    \end{equation}
\end{lem}

Using the estimate \eqref{eq:tHigherintegrability2} with $t=1/q$ and $Q=(Du)_{x_0,2r}$, and following the same argument as in the proof of  \cite[Lemma 3.5]{OSS25},  we further obtain the following estimates.

\begin{lem}\label{lem:Higherintegrability3}
    Let $u\in W^{1,p}(\Omega,\mathbb{R}^N)$ with $H(\cdot,|Du|)\in L^1(\Omega)$ be a minimizer of (\ref{F}) and let $\sigma_1>0$ be from Lemma~\ref{lem:Higherintegrability2}. For any $B_{2r}(x_0)\Subset\Omega$ with (\ref{eq:smallness}) and $r\leq\frac{1}{2}$, there exists $c=c(n,N,p,q,\alpha,[a]_{C^{0,\alpha}},L,\nu)>0$ such that
    \begin{equation*}
    \begin{aligned}
        &\left(\fint_{B_r(x_0)}H_{|(Du)_{x_0,2r}|}(x,|Du-(Du)_{x_0,2r}|)^{1+\sigma_1}\, dx\right)^{\frac{1}{1+\sigma_1}}\\
        &\leq c(H_{B_{2r}(x_0)}^-)_{|(Du)_{x_0,2r}|}\left(\fint_{B_{2r}(x_0)}|Du-(Du)_{x_0,2r}|\, dx\right)+cr^{\alpha_2}H_{B_{2r}(x_0)}^-(|(Du)_{x_0,2r}|),
    \end{aligned}
    \end{equation*}
    where $\alpha_2$ is from the previous lemma.
 \end{lem}

\section{Nondegenerate regime}

Let $u\in W^{1,p}(\Omega,\mathbb{R}^N)$ with $H(\cdot,|Du|)\in L^1(\Omega)$ be a minimizer of (\ref{F}), where  the function $H$ in \eqref{H} satisfies \eqref{a} and  the function $f(x,P)$ in \eqref{F} satisfies \eqref{C1}--\eqref{C6}. In this section, we consider the nondegenerate regime \eqref{eq:smalldecaycondition}.
We define the following excess functionals:
\begin{equation*}
    E(x_0,r,Q):=\fint_{B_r(x_0)}|V_{H_{B_r(x_0)}^-}(Du)-V_{H_{B_r(x_0)}^-}(Q)|^2\, dx,
\end{equation*}
and
\begin{equation}\label{Phi}
    \varPhi(x_0,r,Q):=\frac{E(x_0,r,Q)}{H_{B_r(x_0)}^-(|Q|)}.
\end{equation}
In the case $Q=(Du)_{x_0,r}$, we denote 
$$
E(x_0,r) :=  E(x_0,r,(Du)_{x_0,r})
\quad\text{and}\quad
\varPhi(x_0,r):= \varPhi(x_0,r,(Du)_{x_0,r}).
$$
Note that by (\ref{eq:Vtoshift}) and (\ref{eq:Vmean}),
\begin{equation}\label{meanexcess}
\begin{aligned}
    E(x_0,r)&\sim\fint_{B_r(x_0)}H_{|(Du)_{x_0,r}|}(x_{x_0,r}^-,|Du-(Du)_{x_0,r}|)\, dx\\
    &\sim\fint_{B_r(x_0)}|V_{H_{B_r(x_0)}^-}(Du)-(V_{H_{B_r(x_0)}^-}(Du))_{x_0,r}|^2\, dx.
\end{aligned}
\end{equation}
We find an $\mathcal{A}$-harmonic approximation where
\begin{equation}\label{A}
    \mathcal{A}(Q):=\frac{\partial^2 f(x_{x_0,2r}^-,Q)}{H''(x_{x_0,2r}^-,|Q|)},\qquad Q\in\mathbb{R}^{N\times n}.
\end{equation}
Observe that this $\mathcal{A}$ satisfies the Legendre-Hadamard condition \eqref{eq:LH1} by \eqref{ass:gc} and \eqref{eq:LH}.

We note that \eqref{C5} implies the H\"older continuity of $f(x,P)$ in $x$. Moreover, from this together with \eqref{C1} and \eqref{C3}, we further find that $\partial f(x,P)$ is H\"older continuous in $x$.

\begin{lem}\label{lem:(A4.5)}
 For every $x_1,x_2\in B_\rho(x_0)\subset\Omega$ with $\rho\leq\frac{1}{2}$ and every $P\in\mathbb{R}^{N\times n}$,
    \begin{equation}\label{eq:(A4.5)}
    \begin{aligned}
        &|\partial f(x_1,P)-\partial f(x_2,P)|\\
         &\leq c \Big(|x_1-x_2|^\frac{\beta_0}{2}+|a(x_1)-a(x_2)|^\frac{1}{2}\Big) \big(H'(x_1,|P|)+H'(x_2,|P|)\big)\\
        &\qquad + c |a(x_1)-a(x_2)| |P|^{q-1}.
    \end{aligned}
    \end{equation}
where    $c=c(n,N,p,q,L)>0$.
\end{lem}
\begin{proof}
  We follow the argument introduced in \cite[Appendix A]{S09}. Note that the case when $P={\bf0}$ is trivial since $|\partial f(x,{\bf 0})|=0$ by the first inequality \eqref{C3}.  Fix $x_1,x_2\in B_\rho(x_0)\subset\Omega$ with $\rho\leq\frac{1}{2}$ and $P \in\mathbb{R}^{N\times n}\setminus\{{\bf 0}\}$.  For $Q\in\mathbb{R}^{N\times n}$ with $|Q|\leq\frac{1}{2}|P|$,
    \begin{equation*}
    \begin{aligned}
        &\langle\,  \partial f(x_1,P)-\partial f(x_2,P)\, |\, Q\, \rangle\\
        &=f(x_1,P+Q)-f(x_2,P+Q) + f(x_2,P)-f(x_1,P)\\
        &\qquad+\int_0^1\int_0^1 \Big[ \langle\,  \partial^2 f(x_2,P+stQ)Q\, |\, Q\, \rangle - \langle\,  \partial^2 f(x_1,P+st Q)Q\, |\, Q\, \rangle \Big]\, ds\, dt\\
        &=l_1+l_2.
    \end{aligned}
    \end{equation*}
    By \eqref{C5} and \eqref{C3}, we have that
    \begin{equation*}
        l_1 \leq L|x_1-x_2|^{\beta_0}\big(H(x_1,|P|)+H(x_2,|P|)\big)+L|a(x_1)-a(x_2)||P|^q,
    \end{equation*}
       \begin{equation*}
        l_2 \leq L|Q|^2 \int_0^1\int_0^1H''(x_1,|P+stQ|)+H''(x_2,|P+stQ|)\, ds\, dt .
    \end{equation*}
    We notice from $|Q|\leq\frac{1}{2}|P|$ that 
    \begin{equation*}
        \begin{dcases}
            \int_0^1\int_0^1|P+st Q|^{\gamma-2}\, ds\, dt\leq\left(\frac{1}{2}|P|\right)^{\gamma-2},\qquad \gamma<2,\\
            \int_0^1\int_0^1|P+st Q|^{\gamma-2}\, ds\, dt\leq\left(\frac{3}{2}|P|\right)^{\gamma-2},\qquad \gamma\geq2.
        \end{dcases}        
    \end{equation*}
    Therefore, 
    \begin{equation*}
        |l_2|\leq \max\left\{2^{2-p},\Big(\frac{3}{2}\Big)^{q-2}\right\}L \left(H''(x_1,|P|)+H''(x_2,|P|)\right)|Q|^2.
    \end{equation*}
    Now choose $Q\in\mathbb{R}^{N\times n}$ that satisfies the following :
    \begin{equation*}
        \begin{dcases}
            \langle\,  \partial f(x_1,P)-\partial f(x_2,P)\, |\, Q\, \rangle=| \partial f(x_1,P)-\partial f(x_2,P)||Q|,\\
            |Q|=|P|\left(|x_1-x_2|^{\frac{\beta_0}{2}}+|a(x_1)-a(x_2)|^{\frac{1}{2}}\right)\frac{1}{2(1+\sqrt{2L})}\leq\frac{1}{2}|P|,
        \end{dcases}
    \end{equation*}
    where the last inequality uses the conditions $0\leq a(x)\leq L$ and $\rho\leq\frac{1}{2}$. Combining all the above results, we have \eqref{eq:(A4.5)}.
\end{proof}
Note that in the proof of the previous lemma we do not use \eqref{a} for $H$ and  \eqref{C2} and \eqref{C4} for $f$. 
Now, we show that the function $u- Q(x-x_0)-(u)_{x_0,2r}$ is almost $\mathcal{A}$-harmonic. 
\begin{lem}
There exists  a constant $c=c(n,N,p,q,\alpha,[a]_{C^{0,\alpha}},L,\nu)>0$ such that  for any ball $B_{2r}(x_0)\Subset\Omega$ satisfying \eqref{eq:smallness} with $r\leq\frac{1}{2}$, the following estimate holds:
    \begin{equation}\label{eq:Aapproximation}
    \begin{aligned}
        &\left|\fint_{B_r(x_0)}\langle\, \mathcal{A}(Q)(Du-Q)\, |\, D\zeta\, \rangle\, dx\right|\\
        &\leq c \left(\varPhi(x_0,2r,Q)+\varPhi(x_0,2r,Q)^{\frac{1+\beta_0}{2}}+ r^{\frac{\alpha_1}{2}}\big(1+\varPhi(x_0,2r,Q)\big)^{\frac{q-1}{p}}\right)|Q|
    \end{aligned}
    \end{equation}
    for all $\zeta\in W_0^{1,\infty}(B_r(x_0),\R^N)$ with $\|D\zeta\|_{L^\infty(B_r(x_0))}\le1$, where $\alpha_1$ is from \eqref{alpha1}.
\end{lem}
\begin{proof}
    In this proof, we omit the dependence on $x_0$ in the notation and denote $H_\rho^\pm:=H_{B_\rho(x_0)}^\pm$ and $x_\rho^\pm:=x_{x_0,\rho}^\pm$ for $\rho>0$ (see \eqref{eq:frozenH} and \eqref{eq:pm}). From the definition of $\mathcal{A}(Q)$ in \eqref{A}, we have
    \begin{equation*}
    \begin{aligned}
        &H''(x_{2r}^-,|Q|)\fint_{B_r}\langle\, \mathcal{A}(Q)(Du-Q)\, |\, D\zeta\, \rangle\, dx=\fint_{B_r}\langle\,  \partial^2 f(x_{2r}^-,Q)(Du-Q)\, |\, D\zeta\, \rangle\, dx\\
        &=\fint_{B_r}\int_0^1\langle\,  \big(\partial^2 f(x_{2r}^-,Q)-\partial^2 f(x_{2r}^-,Q+t(Du-Q))\big)(Du-Q)\, |\, D\zeta\, \rangle\, dt\, dx\\
        &\qquad+\fint_{B_r}\int_0^1\langle\,  \partial^2 f(x_{2r}^-,Q+t(Du-Q))(Du-Q)\, |\, D\zeta\, \rangle\, dt\, dx\\
        &=\underbrace{\fint_{B_r}\chi_{F_1}\int_0^1\langle\,  \big(\partial^2 f(x_{2r}^-,Q)-\partial^2 f(x_{2r}^-,Q+t(Du-Q))\big)(Du-Q)\, |\, D\zeta\, \rangle\, dt\, dx}_{=:J_1}\\
        &\qquad+ \underbrace{\fint_{B_r}\chi_{F_2}\int_0^1\langle\,  \big(\partial^2 f(x_{2r}^-,Q)-\partial^2 f(x_{2r}^-,Q+t(Du-Q))\big)(Du-Q)\, |\, D\zeta\, \rangle\, dt\, dx}_{=:J_2}\\
        &\qquad+\underbrace{\fint_{B_r}\int_0^1\langle\,  \partial^2 f(x_{2r}^-,Q+t(Du-Q))(Du-Q)\, |\, D\zeta\, \rangle\, dt\, dx}_{=:J_3},
    \end{aligned}
    \end{equation*}
    where $F_1:=\{|Du-Q|\geq\frac{1}{2}|Q|\}\cap B_r$ and  $F_2:=B_r\setminus F_1$.
    Using \eqref{C3}, Lemma \ref{lem:phi''line} and \eqref{eq:Neqv}, we find that
    \begin{equation*}
    \begin{aligned}
        |J_{1}|&\leq c\fint_{B_r}\chi_{F_1}\int_0^1H''(x_{2r}^-,|Q|)+H''(x_{2r}^-,|Q+t(Du-Q)|)\, dt\ |Du-Q|\, dx\\
        &\leq c\fint_{B_r}\chi_{F_1}\big(H''(x_{2r}^-,|Q|)+H''(x_{2r}^-,|Q|+|Du-Q|)\big)|Du-Q|\, dx\\
        &\leq c\fint_{B_r}\chi_{F_1}H'(x_{2r},|Q|+|Du|)\frac{|Du-Q|}{|Q|}\, dx.
    \end{aligned}
    \end{equation*}
Since we have that $3|Du-Q|\geq|Du-Q|+|Q|$ in $F_1$, by \eqref{eq:shifteqv},
    \begin{equation*}
        H'(x_{2r}^-,|Q|+|Du|)\leq 3\frac{H'(x_{2r}^-,2|Q|+|Du-Q|)}{|Q|+|Du-Q|}|Du-Q|\leq c(H_{|Q|})'(x_{2r}^-,|Du-Q|).
    \end{equation*}
    Thus from \eqref{meanexcess} and \eqref{Phi},
    \begin{equation*}
    \begin{aligned}
        |J_{1}|&\leq\frac{c}{|Q|}\fint_{B_r}\chi_{F_1} (H_{|Q|})'(x_{2r}^-,|Du-Q|)|Du-Q|\, dx\\
        &\leq c\frac{1}{|Q|}E(2r,Q)\leq c|Q|H''(x_{2r}^-,|Q|)\varPhi(2r,Q).
    \end{aligned}
    \end{equation*}

We next  consider $J_2$.   In $F_2$, by \eqref{C6}, 
    \begin{equation*}
        |\partial^2 f(x_{2r}^-,Q)-\partial^2 f(x_{2r}^-,Q+t(Du-Q))|\leq cH''(x_{2r}^-,|Q|)\left(\frac{|Du-Q|}{|Q|}\right)^{\beta_0},
    \end{equation*}
    for every $t\in(0,1]$, so that
    \begin{equation*}
        |J_{2}|\leq c|Q|H''(x_{2r}^-,|Q|)\fint_{B_{r}}\chi_{F_2}\left(\frac{|Du-Q|}{|Q|}\right)^{\beta_0+1}\, dx.
    \end{equation*}
    Moreover, since $|Q|+|Du-Q|\leq\frac{3}{2}|Q|$ in $F_2$,
    \begin{equation*}
        \begin{aligned}
            \frac{|Du-Q|^2}{|Q|^2}=\frac{H'(x_{2r}^-,|Q|)|Du-Q|^2}{H'(x_{2r}^-,|Q|)|Q|^2}&\leq c\frac{H'(x_{2r}^-,|Q|+|Du-Q|)|Du-Q|^2}{H(x_{2r}^-,|Q|)(|Q|+|Du-Q|)}\\
            &\leq\frac{c}{H(x_{2r}^-,|Q|)}H_{|Q|}(x_{2r}^-,|Du-Q|).
        \end{aligned}
    \end{equation*}
    Using Jensen's inequality with the fact that $\frac{1+\beta_0}{2}<1$,
    \begin{equation*}
    \begin{aligned}
        |J_{2}|&\leq c|Q|H''(x_{2r}^-,|Q|)\fint_{B_{r}}\left(\frac{H_{|Q|}(x_{2r}^-,|Du-Q|)}{H(x_{2r}^-,|Q|)}\right)^{\frac{1+\beta_0}{2}}\, dx\\
        &\leq c|Q|H''(x_{2r}^-,|Q|)\varPhi(2r,Q)^{\frac{1+\beta_0}{2}}.
    \end{aligned}
    \end{equation*}
  
 We finally consider $J_3$.  Using the facts that $\int_{B_r}\langle\,  \partial f(x_{2r}^-,Q)\, |\, D\zeta\, \rangle\, dx=0$ and\\
 $\int_{B_r}\langle\,  \partial f(x,Du)\, |\, D\zeta\, \rangle\, dx=0$ (this is the weak formulation of the Euler-Lagrange system of the functional \eqref{F})  and applying \eqref{eq:(A4.5)}, \eqref{eq:H'oderintegrability} and \eqref{eq:qintegrabilityq-1}, we obtain
    \begin{equation*}
        \begin{aligned}
            |J_3|&=\left|\fint_{B_r}\langle\,  \partial f(x_{2r}^-,Du)-\partial f(x,Du)\, |\, D\zeta\, \rangle\, dx\right|\\
            &\leq c \left(r^\frac{\beta_0}{2}+r^\frac{\alpha}{2}\right)\fint_{B_r}H'(x,|Du|)\, dx+\fint_{B_r}|a(x)-a(x_{2r}^-)| |Du|^{q-1}\, dx\\
            &\leq c \left(r^\frac{\beta_0}{2}+r^\frac{\alpha}{2}\right)\fint_{B_r}H'(x_{2r}^-,|Du|)\, dx + c \fint_{B_r}|a(x)-a(x_{2r}^-)||Du|^{q-1}\, dx\\
            &\leq c \left(r^\frac{\beta_0}{2}+r^\frac{\alpha}{2}+r^{\alpha_0}\right) H'\left(x_{2r}^-, \fint_{B_{2r}} |Du|\, dx\right).
        \end{aligned}
    \end{equation*}
    Since the mapping $t\mapsto\frac{H'(x^-_{2r},t)}{t^{p-1}}$ is increasing, recalling the set $F_1=\{|Du-Q|\geq\frac{1}{2}|Q|\}\cap B_r$,
    \begin{equation*}
        \begin{aligned}
            |Du|^p&\leq c\chi_{F_1}|Du-Q|^p+c|Q|^p\\
            &\leq c\chi_{F_1}|Du-Q|^{2}\frac{(|Du-Q|+|Q|)^{p-1}}{|Du-Q|+|Q|}\cdot\frac{|Q|^{p-1}}{|Q|^{p-1}}+c|Q|^p\\
            &\leq c\chi_{F_1}|Du-Q|^{2}\frac{H'(x_{2r}^-,|Du-Q|+|Q|)}{|Du-Q|+|Q|}\cdot\frac{|Q|^{p-1}}{H'(x_{2r}^-,|Q|)}+c|Q|^p\\
            &\leq c \left(\chi_{F_1}H_{|Q|}(x_{2r}^-,|Du-Q|)\cdot\frac{1}{H(x_{2r}^-,|Q|)}+1\right)|Q|^p,
        \end{aligned}
    \end{equation*}
    whence with (\ref{eq:H'oderintegrability}) and H\"older's inequality,
    \begin{equation*}
        \begin{aligned}
            |J_3|&\leq c r^{\frac{\alpha_1}{2}}H'\left(x_{2r}^-,\left(\fint_{B_{2r}}|Du|^p\, dx\right)^\frac{1}{p}\right)\\
            &\leq c r^{\frac{\alpha_1}{2}} H'\left(x_{2r}^-,|Q|\left(\frac{1}{H(x_{2r}^-,|Q|)}\fint_{B_2r}H_{|Q|}(x_{2r}^-,|Du-Q|)\, dx+1\right)^{\frac{1}{p}}\right)\\
            &\leq c r^{\frac{\alpha_1}{2}} H'(x_{2r}^-,|Q|)\big(\varPhi(2r,Q)+1\big)^\frac{q-1}{p}\\
            &\leq c r^{\frac{\alpha_1}{2}} H''(x_{2r}^-,|Q|)|Q|\big(\varPhi(2r,Q)+1\big)^\frac{q-1}{p}.
        \end{aligned}
    \end{equation*}
  Collecting the above estimates completes the proof.
  \end{proof}

The previous lemma, together with the higher integrability result in Lemma \ref{lem:Higherintegrability3}, allows us to apply Lemma \ref{Aharmonicapproximation}.
    Define
    \begin{equation}\label{alpha3}
        \alpha_3:=\min\left\{\frac{\alpha_1}{2},\alpha_2,\frac{\beta_0q}{q-1}\right\}=\min\left\{\frac{\alpha}{2},\frac{\alpha_0}{2}, \frac{\beta_0}{2}, \frac{\beta_0q}{q-1},\alpha-\frac{(q-p)n}{p(1+\sigma_0)}\right\},
    \end{equation}
    where $\alpha_1$ and $\alpha_2$ are from \eqref{alpha1} and \eqref{alpha2}, respectively.
We then set
\begin{equation}\label{excess*}
    \begin{aligned}
        E_*(x_0,\rho)&:=E(x_0,\rho)+\rho^{\frac{\alpha_3}{2}}H_{B_\rho(x_0)}^-(|(Du)_{x_0,\rho}|)\\
        &=H_{B_\rho(x_0)}^-(|(Du)_{x_0,\rho}|)\left(\varPhi(x_0,\rho)+\rho^{\frac{\alpha_3}{2}}\right).
    \end{aligned}
\end{equation} Then we have the excess decay in the nondegenerate regime.

\begin{lem}\label{lem:ndegexcessdecay0}
    For every $\varepsilon>0$, there exist $\delta_1,\delta_2\in(0,1)$ depending on $n,N,p,q,\alpha,[a]_{C^{0,\alpha}},L,\nu,\beta_0$ and $\varepsilon$ such that for any ball $B_{2r}(x_0)\Subset \Omega$ satisfying (\ref{eq:smallness}) with $r\leq\frac{1}{2}$, if
    \begin{equation}\label{eq:smalldecaycondition}
        \begin{aligned}
            \fint_{B_{2r}(x_0)}\left|V_{H_{B_{2r}(x_0)}^-}(Du)-(V_{H_{B_{2r}(x_0)}^-}(Du))_{x_0,2r}\right|^2\, dx\leq \delta_1\fint_{B_{2r}(x_0)}\left|V_{H_{B_{2r}(x_0)}^-}(Du)\right|^2\, dx
        \end{aligned}
    \end{equation}
    and
    \begin{equation}\label{eq:smallradiuscondition}
        r^\frac{\alpha_3}{2}\leq\delta_2,
    \end{equation}
    then we have that for every $\tau\in(0,\frac{1}{4})$,
    \begin{equation*}
        \fint_{B_{\tau r}(x_0)}\left|V_{H_{B_{\tau r}(x_0)}^-}(Du)-(V_{H_{B_{\tau r}(x_0)}^-}(Du))_{x_0,\tau r}\right|^2\, dx\leq c\tau^2\left(1+\frac{\varepsilon}{\tau^{n+2}}\right)E_*(x_0,2r)
    \end{equation*}
    for some $c=c(n,N,p,q,\alpha,[a]_{C^{0,\alpha}},L,\nu,\beta_0)>0$.
\end{lem}
    \begin{proof}
        We omit the dependence on $x_0$ from notation, and write $H_\rho^\pm:=H_{B_\rho}^\pm$ and $V_\rho^\pm:=V_{H_\rho^\pm}$ for $\rho>0$. Set $Q:=(Du)_{2r}$.
       First, by (\ref{eq:smalldecaycondition}), together with  \eqref{eq:Vmean}  and  (\ref{Phi}), and by choosing $\delta_1\in(0,1)$ small enough, we obtain
        \begin{equation*}
            \begin{aligned}
                \fint_{B_{2r}}|V_{2r}^-(Du)|^2\, dx&\leq 2\fint_{B_{2r}}|V_{2r}^-(Du)-V_{2r}^-(Q)|^2\, dx+2|V_{2r}^-(Q)|^2\\
                &\leq c \fint_{B_{2r}}|V_{2r}^-(Du)-\left(V_{2r}^-(Du)\right)_{2r}|^2\, dx+2|V_{2r}^-(Q)|^2\\
                &\leq c\delta_1\fint_{B_{2r}}|V_{2r}^-(Du)|^2\, dx+2|V_{2r}^-(Q)|^2\\
                &\leq \frac12 \fint_{B_{2r}}|V_{2r}^-(Du)|^2\, dx+2|V_{2r}^-(Q)|^2,
            \end{aligned}
        \end{equation*}
which, together with (\ref{eq:Vsquare}), implies 
       \begin{equation}\label{eq:ugradientestimate}
            \fint_{B_{2r}}H_{2r}^-(|Du|)\, dx\leq cH_{2r}^-(|Q|),
        \end{equation}   
        and
        \begin{equation}\label{Phidelta1}
        \varPhi(2r)\le c\delta_1 \le \frac12.
        \end{equation}    
              Moreover, (\ref{eq:Aapproximation}),  \eqref{Phidelta1}, \eqref{excess*}, and  \eqref{eq:smallradiuscondition} imply that
        \begin{equation}\label{eq:smallapproximation}
            \begin{aligned}
                &\bigg|\fint_{B_r}\langle\, \mathcal{A}(Q)(Du-Q)\, |\, D\zeta\, \rangle\, dx\bigg|\\
                &\leq c\left(\varPhi(2r)+\varPhi(2r)^\frac{1+\beta_0}{2}+r^{\alpha_3}\big(1+\varPhi(2r)\big)^\frac{q-1}{p}\right)|Q|\\
                &\leq c\left(\varPhi(2r)^\frac{1}{2}+\varPhi(2r)^\frac{\beta_0}{2}+r^\frac{\alpha_3}{2}\right)\left(\varPhi(2r)+r^{\alpha_3} \right)^\frac{1}{2}|Q|\\
                &\leq \Tilde{c}_1\left(\delta_1^\frac{1}{2}+\delta_1^\frac{\beta_0}{2}+\delta_2\right)\left(\frac{E_*(2r)}{H_{2r}^-(|Q|)}\right)^\frac{1}{2}|Q|
            \end{aligned}
        \end{equation}
        for all $\zeta\in W_0^{1,\infty}(B_r,\R^N)$ with $\|D\zeta\|_{L^\infty(B_r)}\le 1$ with some $\Tilde{c}_1\ge1$.        
        
        Define an N-function $\Psi$ by
        \begin{equation*}
            \Psi(t):=\frac{H_{|Q|}(x_{2r}^-,t)}{H_{2r}^-(|Q|)}\sim\frac{H_{2r}^-(|Q|+t)}{H_{2r}^-(|Q|)}\frac{t^2}{(|Q|+t)^2},\quad t\geq0.
        \end{equation*}
        We note that for $t\in[0,|Q|]$,
        \begin{equation*}
            \left(\frac{t}{|Q|}\right)^2\leq \frac{4}{(|Q|+t)^2}t^2\leq 4\frac{H_{2r}^-(|Q|+t)}{H_{2r}^-(|Q|)(|Q|+t)^2}t^2\leq \Tilde{c}_2\Psi(t)
        \end{equation*}
        for some $\Tilde{c}_2\ge 1$, and that by Lemma \ref{lem:Higherintegrability3}, (\ref{meanexcess}) and (\ref{excess*}),
        \begin{equation*}
        \begin{aligned}
            &\left(\fint_{B_r}\Psi(|Du-Q|)^{1+\sigma_1}\, dx\right)^\frac{1}{1+\sigma_1}\\
            &=\frac{1}{H_{2r}^-(|Q|)}\left(\fint_{B_{r}}H_{|Q|}(x_{2r}^-,|Du-Q|)^{1+\sigma_1}\, dx\right)^\frac{1}{1+\sigma_1}\\
            &\leq \frac{c}{H_{2r}^-(|Q|)}\left(\fint_{B_{2r}}H_{|Q|}(x_{2r}^-,|Du-Q|)\, dx+r^{\alpha_2}H_{2r}^-(|Q|)\right)\leq \Tilde{c}_3\frac{E_*(2r)}{H_{2r}^-(|Q|)}
        \end{aligned}
        \end{equation*}
        for some constant $\Tilde{c}_3\ge 1$.
        Then we define
        \begin{equation}\label{def:mu}
            \mu:=\max\left\{\Tilde{c}_1 ,\sqrt{\Tilde{c}_2\Tilde{c}_3(2\Tilde{c}_4)^{\Tilde{q}}}\right\}\left(\frac{E_*(2r)}{H_{2r}^-(|Q|)}\right)^\frac{1}{2}|Q|,
        \end{equation}
       where $\tilde{c}_4\ge 1$ is a constant depending only on $n$, $N$, $p$ and $q$, to be determined later, and $\tilde{q}:=\max\{q,2\}$.
       Then, since $\frac{E_*(2r)}{H_{2r}^-(|Q|)} \lesssim  \delta_1+\delta_2$ by (\ref{eq:smalldecaycondition}) and (\ref{eq:smallradiuscondition}), by choosing $\delta_1$ and $\delta_2$ sufficiently small, we obtain that
        \begin{equation}\label{eq:smallmu}
            \mu \leq|Q|
        \end{equation}
 and hence
        \begin{equation}\label{eq:zetaHigherintegrability}
            \left(\fint_{B_r}\Psi(|Du-Q|)^{1+\sigma_1}\, dx\right)^\frac{1}{1+\sigma_1}\leq\frac{1}{\Tilde{c}_2(2\Tilde{c}_4)^{\Tilde{q}}}\left(\frac{\mu}{|Q|}\right)^2\leq\frac{1}{(2\Tilde{c}_4)^{\Tilde{q}}}\Psi(\mu).
        \end{equation}
        We then determine $\delta>0$ from Lemma \ref{Aharmonicapproximation} with given $\varepsilon$ and $\Psi=\varphi$, and choose $\delta_1$ and $\delta_2$ small enough to satisfy
        \begin{equation*}
            \delta_1^\frac{1}{2}+\delta_1^\frac{\beta_0}{2}+\delta_2\leq\delta
        \end{equation*}
        which applies to (\ref{eq:smallapproximation}) as
        \begin{equation*}
            \fint_{B_r}\langle\, \mathcal{A}(Q)(Du-Q)\, |\, D\zeta\, \rangle\, dx\leq\frac{\Tilde{c}_1(\delta_1^\frac{1}{2}+\delta_1^\frac{\beta_0}{2}+\delta_2)}{\max\left\{\Tilde{c}_1,\sqrt{\Tilde{c}_2\Tilde{c}_3(2\Tilde{c}_4)^q}\right\}}\mu \leq \delta\mu .
        \end{equation*}
        Define the affine function $\ell(x):=(Du)_{2r}(x-x_0)+ (u)_{2r}$. Then we may apply Lemma \ref{Aharmonicapproximation} with $(w,\phi,\sigma)$ replaced by $(u-\ell,\Psi,\sigma_1)$ to obtain
        \begin{equation*}
            \frac{1}{H_{2r}^-(|Q|)}\fint_{Br}H_{|Q|}(x_{2r}^-,|Du-Q-Dh|)\, dx\leq \varepsilon\Psi(\mu),
        \end{equation*}
        where $h$ is the $\mathcal{A}$-harmonic function in $B_r$ with $h=u-\ell $ on $\partial B_r$. Moreover, by (\ref{eq:smallmu}) and \eqref{def:mu},
        \begin{equation*}
            \Psi(\mu)\leq c\frac{H_{2r}^-(|Q|+\mu)}{H_{2r}^-(|Q|)}\frac{\mu^2}{(|Q|+\mu)^2}\leq c\frac{H_{2r}^-(2|Q|)}{H_{2r}^-(|Q|)}\left(\frac{\mu}{|Q|}\right)^2\leq c\left(\frac{\mu}{|Q|}\right)^2\leq c \frac{E_*(2r)}{H_{2r}^-(|Q|)},
        \end{equation*}
        and hence
        \begin{equation}\label{eq:shiftharmonicbound}
            \fint_{B_r}H_{|Q|}(x_{2r}^-,|Du-Q-Dh|)\, dx\leq c \varepsilon E_*(2r).
        \end{equation}
        
        Note that by \cite{Giusti}, $h$ is smooth in the interior of $B_r$ and satisfies, in particular, the following estimates:
        \begin{equation}\label{eq:Aharmonicgradient1}
         r\sup_{B_{\frac{r}{4}}}|D^2h| \leq c \sup_{B_{\frac{r}{2}}}|Dh| \leq c\fint_{B_r}|Dh|\, dx
        \end{equation}
        and Calder\'on-Zygmund type estimate
        \begin{equation}\label{eq:Aharmonicgradient2}
            \fint_{B_r}\Psi(|Dh|)\, dx\leq c\fint_{B_r}\Psi(|Du-Q|)\, dx,
        \end{equation}
        which, with Jensen's inequality, give
        \begin{equation*}
             \sup_{B_\frac{r}{2}}|Dh|\leq c\Psi^{-1}\left(\fint_{B_r}\Psi(|Dh|)\, dx\right)\leq\Tilde{c}_4\Psi^{-1}\left(\fint_{B_r}\Psi(|Du-Q|)\, dx\right)
        \end{equation*}
        for some $\Tilde{c}_4=\Tilde{c}_4(n,N,p,q)\geq1$.  Then by (\ref{eq:zetaHigherintegrability}) and (\ref{eq:smallmu}), we have
        \begin{equation}\label{eq:Aharmonicgradient3}
            \sup_{B_\frac{r}{2}}|Dh|\leq \Tilde{c}_5\Psi^{-1}\left(\frac{1}{(2\Tilde{c}_5)^{\Tilde{q}}}\Psi(\mu)\right)\leq \Tilde{c}_5\Psi^{-1}\left(\Psi\left(\frac{1}{2\Tilde{c}_5}\mu\right)\right)\leq\frac{1}{2}|Q|.
        \end{equation}
        
        Fix $\tau\in(0,\frac{1}{4})$. Then by (\ref{eq:Aharmonicgradient3}), and $\frac{1}{2}|Q|\leq|Q+(Dh)_{\tau r}|\leq\frac{3}{2}|Q|$ so that with (\ref{eq:Vtoshift}),
        \begin{equation*}
            \begin{aligned}
                &\fint_{B_{\tau r}}\left|V_{\tau r}^-(Du)-\left(V_{\tau r}^-(Du)\right)_{\tau r}\right|^2\, dx\\
                &\leq c\fint_{B_{\tau r}}\left|V_{\tau r}^-(Du)-V_{\tau r}^-(Q+(Dh)_{\tau r})\right|^2\, dx\\
                &\leq c\fint_{B_{\tau r}}H_{|Q+(Dh)_{\tau r}|}(x_{\tau r}^-,|Du-Q-(Dh)_{\tau r}|)\, dx\\
                &\leq c\fint_{B_{\tau r}}H_{|Q|}(x_{\tau r}^-,|Du-Q-(Dh)_{\tau r}|)\, dx\\
                &\leq c\underbrace{\fint_{B_{\tau r}}H_{|Q|}(x_{\tau r}^-,|Du-Q-Dh|)-H_{|Q|}(x_{2 r}^-,|Du-Q-Dh|)\, dx}_{=:K_1}\\
                &\quad+c\underbrace{\fint_{B_{\tau r}}H_{|Q|}(x_{2 r}^-,|Du-Q-Dh|)\, dx}_{=:K_2}+c\underbrace{\fint_{B_{\tau r}}H_{|Q|}(x_{\tau r}^-,|Dh-(Dh)_{\tau r}|)\, dx}_{=:K_3}.
            \end{aligned}
        \end{equation*}
        For $K_1$, we use \eqref{eq:Aharmonicgradient3}, \eqref{eq:qintegrabilityq},  (\ref{eq:ugradientestimate}), \eqref{eq:smallradiuscondition} by choosing $\delta_2\leq\varepsilon$, and \eqref{excess*} to find
        \begin{equation*}
            \begin{aligned}
                K_1&\leq c\fint_{B_{\tau r}}(a(x)-a_{2r}^-)(|Du|^q+|Q|^q+|Dh|^q)\, dx\\
                &\leq c\tau^{-n}\fint_{B_{r}}(a(x)-a_{2r}^-)|Du|^q\, dx+cr^\alpha|Q|^q\\
                &\leq cr^{\alpha_0}\tau^{-n}(H_{2r}^-(|Du|))_{2r}+cr^\alpha(|Du|^{p(1+{\sigma_0})})_{2r}^{\frac{q-p}{p(1+\sigma_0)}}|Q|^p\\
                &\leq cr^{\alpha_0}\tau^{-n}(H_{2r}^-(|Du|))_{2r}+r^{\alpha_2}H_{2r}^-(|Q|)\\
                &\leq cr^{\alpha_3}\tau^{-n}H_{2r}^-(|Q|)=cr^\frac{\alpha_3}{2}\tau^{-n}H_{2r}^-(|Q|)r^\frac{\alpha_3}{2}\\
                &\leq c\varepsilon\tau^{-n}E_*(2r),
            \end{aligned}
        \end{equation*}
        where we also used (\ref{eq:smallness}) to obtain the exponent $\alpha_2$. 
        For $K_2$, we have from  (\ref{eq:shiftharmonicbound}) that
        \begin{equation*}
            K_2\leq c\varepsilon\tau^{-n}E_*(2r).
        \end{equation*} 
        Finally, to estimate $K_3$, we use (\ref{eq:Aharmonicgradient1}) and (\ref{eq:Aharmonicgradient2}) 
with $\Psi$ replaced by the function $t \mapsto t^{\tilde p}$ with $\tilde p := \min\{p,2\}$, to obtain
\begin{equation*}
    \begin{aligned}
        K_3
        &\le c H_{|Q|}(x_{\tau r}^-,\, \tau r \sup_{B_{r/4}} |D^2 h|) \\
        &\le c H_{|Q|}(x_{\tau r}^-,\, \tau  \sup_{B_{r/2}} |D h|) \\
        &\le c \tau^2 H_{|Q|}(x_{2r}^+,\,  \sup_{B_{r/2}} |D h|) \\
        &\le c \tau^2 H_{|Q|}(x_{2r}^+,\, (|Dh|^{\tilde p})_r^{1/\tilde p}) \\
        &\le c \tau^2 H_{|Q|}(x_{2r}^+,\, (|Du - Q|^{\tilde p})_r^{1/\tilde p}),
    \end{aligned}
\end{equation*}
where the third inequality follows from \eqref{eq:Aharmonicgradient3} and the fact that, for $0 < t \le |Q|$ and $\tau \in (0,1)$,
\begin{equation*}
    H_{|Q|}(x_{2r}^+,\, \tau t)
    \sim
    \frac{H'(x_{2r}^+,\, |Q| + \tau t)}{|Q| + \tau t} (\tau t)^2
    \sim
    \tau^2 \frac{H'(x_{2r}^+,\, |Q| + t)}{|Q| + t} t^2
    \sim
    \tau^2 H_{|Q|}(x_{2r}^+,\, t).
\end{equation*}
       Moreover, by applying Jensen's inequality to the convex function  $t\mapsto H_{|Q|}(x_{2 r}^-,t^\frac{1}{\Tilde{p}})$, \eqref{eq:Hoderintegrability} and \eqref{eq:smallness}, and recalling $Q=(Du)_{2r}$, \eqref{meanexcess} and \eqref{excess*}, we have
        \begin{equation*}
        \begin{aligned}
            K_3&\leq c\tau^2\left(H_{|Q|}\left(x_{2 r}^-,(|Du-Q|^{\Tilde{p}})_r^\frac{1}{\Tilde{p}}\right)+r^\alpha\left((|Du-Q|^{\Tilde{p}})_{2r}^\frac{1}{\Tilde{p}}+|Q|\right)^q\right)\\
            &\leq c\tau^2\left(\fint_{B_{2r}}H_{|Q|}(x_{2 r}^-,|Du-Q|)\, dx+r^\alpha\left((|Du|^{\Tilde{p}})_{2r}^\frac{1}{\Tilde{p}}+|Q|\right)^q\right)\\
            &\leq c\tau^2\left(E(2r)+r^\alpha|Q|^q\right)\\
            &\leq c\tau^2\left(E(2r)+r^\alpha|Q|^{q-p}H(x_{2r}^-,|Q|)\right)\\
            &\leq c\tau^2\left(E(2r)+r^\alpha\left((|Du|^{p(1+\sigma_0)})_{2r}\right)^\frac{q-p}{p(1+\sigma_0)}H(x_{2r}^-,|Q|)\right)\\
            &\leq c\tau^2\left(E(2r)+r^{\alpha_3}H(x_{2r}^-,|Q|)\right)\\
            &\leq c\tau^2E_*(2r).
        \end{aligned}
        \end{equation*}
            
        Consequently, collecting the estimates for $K_1$, $K_2$, and $K_3$, we have the assertion.
    \end{proof}

\section{Degenerate Regime}

Let $u\in W^{1,p}(\Omega,\mathbb{R}^N)$ with $H(\cdot,|Du|)\in L^1(\Omega)$ be a minimizer of (\ref{F}), where  the function $H$ in \eqref{H} satisfies \eqref{a} and  the function $f(x,P)$ in \eqref{F} satisfies \eqref{C1}--\eqref{C5} and \eqref{C7}. In this subsection, we will consider the degenerate regime \eqref{eq:degeneratedecaay}.
To handle the degenerate regime, we first establish the almost minimizing estimate in Lemma~\ref{lem:minimizingapproximation} with $w=u$ and $\phi=H^-_{B_{2r}(x_0)}$. Then the comparison estimate in Lemma~\ref{lem:minimizingapproximation}  allows us to derive the excess decay.

\begin{lem}\label{lem:almostminimizer} 
There exists $c=c(n,N,p,q,\alpha,[a]_{C^{0,\alpha}},L,\nu,\beta_0)>0$ such that for any  $\delta>0$ and $B_{2r}(x_0)\Subset\Omega$ with (\ref{eq:smallness}) and $r\leq\frac{1}{2}$,
\begin{equation}\label{eq:almostminimizer}
\begin{aligned}
    \fint_{B_r(x_0)}H^-_{B_r(x_0)}&(|Du|)\, dx\leq\fint_{B_r(x_0)}H^-_{B_r(x_0)}(|Du+D\zeta|)\, dx\\
    & +c\left(\frac{(H^-_{B_{2r}(x_0)}(|Du|))_{x_0,2r}^{\sigma_0}}{H^-_{B_r(x_0)}(\kappa)^{\sigma_0}}+r^{\alpha_2}+\delta\right) \left(\frac{\|D\zeta\|_{\infty}}{(H^-_{B_r(x_0)})^{-1}(\mu)}+1\right)^\gamma \mu\\
\end{aligned}
\end{equation}
for all $\zeta\in W_0^{1,\infty}(B_r(x_0))$ with $\|D\zeta\|_\infty:= \|D\zeta\|_{L^\infty(B_{r}(x_0);\R^{N\times n}))}$, where $\sigma_0$ is given in Lemma~\eqref{lem:Highertintegrability0},
\begin{equation*}
    \mu:= \fint_{B_{2r}(x_0)}H^-_{B_{2r}(x_0)}(|Du|)\, dx ,
\end{equation*}
$\gamma:=q(1+\sigma_0)$, $\alpha_2$ is from \eqref{alpha2}, and $\kappa>0$ is from (\ref{ass:A6}) for given $\delta>0$.
\end{lem}
\begin{proof}
We omit the dependence on $x_0$ from the notation, and write $H^{\pm}(t):=H_{B_{r}(x_0)}^\pm(t)$, $f^-(t):=f(x_{x_0,r}^-,t)$ and $a^\pm:=a(x_{x_0,r}^\pm)$. 
Set $J:=(H^+)^{-1}(\mu)$. Note that, by (\ref{eq:smallness}), 
$$
J^p\leq \mu \le \left(\fint_{B_{2r}}H^-_{B_{2r}}(|Du|)^{1+\sigma_0}\, dx\right)^{\frac{1}{1+\sigma_0}}   \leq |B_{2r}|^{-\frac{1}{1+\sigma_0}},
$$
hence
\begin{equation}\label{rJestimate}
            (a^+-a^-)J^q\leq cJ^pr^\alpha J^{q-p}\leq cJ^pr^{\alpha-\frac{(q-p)n}{p(1+\sigma_0)}},
\end{equation}
\begin{equation*}
    \mu=H^+(J) = H^-(J)+  (a^+-a^-)J^q \leq \left(1+r^{\alpha-\frac{(q-p)n}{p(1+\sigma_0)}}\right)H^-(J),
\end{equation*}
which implies
\begin{equation}\label{Jsimmilar}
    \mu=H^+(J) \sim H^-(J)
    \quad \text{and}\quad
    J= (H^+)^{-1}(\mu) \sim (H^-)^{-1}(\mu).
\end{equation}

We observe 
that 
\begin{equation*}
    \begin{aligned}
        \fint_{B_r}H^-(|Du|)\, dx&\leq\underbrace{\fint_{B_r}H^-(|Du|)-f^-(Du)+f^-(\mathbf{0})\, dx}_{=:L_1}\\
        &\quad+\underbrace{\fint_{B_r}f^-(Du+D\zeta)-f^-(\mathbf{0})-H^-(|Du+D\zeta|)\, dx}_{=:L_2}\\
        &\quad+\underbrace{\fint_{B_r}f^-(Du)-f^-(Du+D\zeta)\, dx}_{=:L_3}\\
        &\quad+\fint_{B_r}H^-(|Du+D\zeta|)\, dx.
    \end{aligned}
\end{equation*}
We then estimate the integrals $L_i$ $(i=1,2,3)$ separately.

\textit{Estimate of $L_1$.} Let $\kappa>0$ be from (\ref{ass:A6}) for a given $\delta>0$ and set 
\begin{equation*}
    E_1:=\{|Du|<\kappa\}\cap B_r
    \quad \text{and}\quad
    E_2:=\{|Du|\geq\kappa\}\cap B_r.
\end{equation*}
Then by (\ref{ass:A6}),
\begin{equation*}
    \fint_{B_r}\chi_{E_1}|H^-(|Du|)-f^-(Du)+f^-(\mathbf{0})|\, dx\leq\delta\fint_{B_r}H^-(|Du|)\, dx\leq c\delta \mu.
\end{equation*}
On the other hand, by the growth condition (\ref{ass:gc}) and \eqref{eq:Hoderintegrability} with Jensen's inequality,
\begin{equation*}
\begin{aligned}
    &\fint_{B_r}\chi_{E_2}|H^-(|Du|)-f^-(Du)+f^-(\mathbf{0})|\, dx\\
    &\leq  c\fint_{B_r}\chi_{E_2}H^-(|Du|)\, dx\\
    &\leq\frac{c}{H^-(\kappa)^{\sigma_0}}\fint_{B_r}\chi_{E_2}H^-(|Du|)^{1+\sigma_0}\, dx\\
    &\leq\frac{c}{H^-(\kappa)^{\sigma_0}}\left(\fint_{B_{2r}}H^-_{B_{2r}}(|Du|)\, dx\right)^{1+\sigma_0}\\
    &\leq c \theta \mu.
\end{aligned}
\end{equation*}
where $\theta:=\frac{\left(\fint_{B_{2r}}H^-_{B_{2r}}(|Du|)\, dx\right)^{\sigma_0}}{H^-(\kappa)^{\sigma_0}}$.
Thus
\begin{equation*}
    |L_1|\leq\fint_{B_r}|H^-(|Du|)-f^-(Du)+f^-(\mathbf{0})|\, dx\leq c(\delta +\theta) \mu .
\end{equation*}

\textit{Estimate of $L_2$.} Let $\kappa>0$ be as before and set
\begin{equation*}
    E_3:=\{|Du+D\zeta|<\kappa\}\cap B_r\ \text{ and }\ E_4:=\{|Du+D\zeta|\geq\kappa\}\cap B_r.
\end{equation*}
Then by (\ref{ass:A6}),
\begin{equation*}
\begin{aligned}
    \fint_{B_r}\chi_{E_3}|H^-(|Du+D\zeta|)-f^-(Du+D\zeta)+f^-(\mathbf{0})|\, dx 
    &\leq\delta\fint_{B_r}H^-(|Du+D\zeta|)\, dx \\
    &\leq c\delta\mu +c\delta\fint_{B_r}H^-(|D\zeta|)\, dx.
\end{aligned}
\end{equation*}
Moreover, using the fact that $H(x,ts)\leq t^qH(x,s)$ for $x\in B_r$, $t\geq1$ and $s>0$ and (\ref{Jsimmilar}),
\begin{equation}\label{eq:etabound}
        \fint_{B_r}H^-(|D\zeta|)\, dx=\fint_{B_r}H^-\left(\frac{|D\zeta|}{J}J\right)\, dx \leq c\left(\frac{\|D\zeta\|_\infty}{(H^-)^{-1}(\mu)}+1\right)^q\mu.
\end{equation}
On the other hand, by the growth condition (\ref{ass:gc}),
\begin{equation*}
\begin{aligned}
    &\fint_{B_r}\chi_{E_4}|H^-(|Du+D\zeta|)-f^-(Du+D\zeta)+f^-(\mathbf{0})|\, dx\\
    &\qquad\leq\frac{c}{H^-(\kappa)^{\sigma_0}}\fint_{B_r}\chi_{E_4}H^-(|Du+D\zeta|)^{1+\sigma_0}\, dx\\
    &\qquad\leq\frac{c}{H^-(\kappa)^{\sigma_0}}\fint_{B_r}H^-(|Du|)^{1+\sigma_0}\, dx+\frac{c}{H^-(\kappa)^{\sigma_0}}\fint_{B_r}H^-(|D\zeta|)^{1+\sigma_0}\, dx.
\end{aligned}
\end{equation*}
For the integrals on the right-hand side, we use \eqref{eq:Hoderintegrability} with Jensen's inequality to obtain 
\begin{equation*}
    \frac{1}{H^-(\kappa)^{\sigma_0}}\fint_{B_r}H^-(|Du|)^{1+\sigma_0}\, dx
    \leq\frac{c}{H^-(\kappa)^{\sigma_0}}\left(\fint_{B_{2r}}H^-_{B_{2r}}(|Du|)\, dx\right)^{1+\sigma_0}
    \leq c \theta \mu.
\end{equation*}
and, also using  \eqref{Jsimmilar},
 \begin{equation*}
    \begin{aligned}
        \frac{1}{H^-(\kappa)^{\sigma_0}}\fint_{B_r}H^-(|D\zeta|)^{1+\sigma_0}\, dx&=\frac{1}{H^-(\kappa)^{\sigma_0}}\fint_{B_r}H^-\left(\frac{|D\zeta|}{J}J\right)^{1+\sigma_0}\, dx\\
        &\leq c\frac{1}{H^-(\kappa)^{\sigma_0}}\left(\frac{\|D\zeta\|_\infty}{(H^-)^{-1}(\mu)}+1\right)^{q(1+\sigma_0)}\mu^{1+\sigma_0}\\
        &= c\theta\left(\frac{\|D\zeta\|_\infty}{(H^-)^{-1}(\mu)}+1\right)^{q(1+\sigma_0)}\mu.
    \end{aligned}
\end{equation*}
Therefore we have
\begin{equation*}
\begin{aligned}
    L_2&=\fint_{B_r}|f^-(Du+D\zeta)-f^-(\mathbf{0})-H^-(|Du+D\zeta|)|\, dx\\
    &\leq c\left\{\delta+\delta\left(\frac{\|D\zeta\|_\infty}{(H^-)^{-1}(\mu)}+1\right)^q+\theta+\theta\left(\frac{\|D\zeta\|_\infty}{(H^-)^{-1}(\mu)}+1\right)^{q(1+\sigma_0)}\right\} \mu.
\end{aligned}
\end{equation*}

\textit{Estimate of $L_3$.} Since $u$ minimizes (\ref{F}), using  \eqref{C5},  we have
\begin{equation*}
    \begin{aligned}
        L_3&\le\fint_{B_r}f^-(Du)-f(x,Du)\, dx+\fint_{B_r}f(x,Du+D\zeta)-f^-(Du+D\zeta)\, dx\\
        &\leq c\fint_{B_r}r^{\beta_0}\big(H(x,|Du|)+H(x,|Du+D\zeta|)\big)+(a(x)-a^-)\big(|Du|^q+|Du+D\zeta|^q\big)\, dx\\
        &\leq cr^{\beta_0}\fint_{B_{r}}H(x,|Du|)\, dx+c\fint_{B_{r}}(a(x)-a^-)|Du|^q\, dx+cr^{\beta_0}\fint_{B_r}H(x,|D\zeta|)\, dx \\
        &\qquad +c\fint_{B_r}(a(x)-a^-)|D\zeta|^q\, dx.
    \end{aligned}
\end{equation*}
By \eqref{eq:Hoderintegrability} with H\"older's inequality and \eqref{eq:qintegrabilityq} with Jensen's inequality,
\begin{equation*}
\begin{aligned}
r^{\beta_0}\fint_{B_{r}}H(x,|Du|)\, dx+\fint_{B_{r}}&(a(x)-a^-)|Du|^q\, dx\\
& \le c r^{\min\{\alpha_0,\beta_0\}} \fint_{B_{2r}}H^-_{B_{2r}}(|Du|)\, dx = c r^{\min\{\alpha_0,\beta_0\}}\mu.
\end{aligned}\end{equation*}
Moreover, as in (\ref{eq:etabound}), adopting (\ref{rJestimate}) and \eqref{Jsimmilar}, we obtain that 
\begin{equation*}
        \fint_{B_r}H(x,|D\zeta|)\, dx=\fint_{B_r}H\left(x,\frac{|D\zeta|}{J}J\right)\, dx\leq c\left(\frac{\|D\zeta\|_\infty}{(H^-)^{-1}(\mu)}+1\right)^q \mu,
\end{equation*}
and
\begin{equation*}
    \begin{aligned}
        \fint_{B_r}(a(x)-a^-)|D\zeta|^q\, dx&\leq\fint_{B_r}(a(x)-a^-)\frac{|D\zeta|^q}{J^q}J^q\, dx\\
        &\leq c\left(\frac{\|D\zeta\|_\infty}{(H^-)^{-1}(\mu)}+1\right)^q\fint_{B_r}(a(x)-a^-)J^q\, dx\\
        &\leq c\left(\frac{\|D\zeta\|_\infty}{(H^-)^{-1}(\mu)}+1\right)^qr^{\alpha-\frac{(q-p)n}{p(1+\sigma_0)}}J^p\\
        &\le c r^{\alpha-\frac{(q-p)n}{p(1+\sigma_0)}}\left(\frac{\|D\zeta\|_\infty}{(H^-)^{-1}(\mu)}+1\right)^q \mu.
    \end{aligned}
\end{equation*}
Thus, combining the last four estimates, we find that
\begin{equation*}
        |L_3| \leq c r^{\min\{\alpha_0,\beta_0,\alpha-\frac{(q-p)n}{p(1+\sigma_0)}\}}\left(\frac{\|D\zeta\|_\infty}{(H^-)^{-1}(\mu)}+1\right)^q \mu.
\end{equation*}
Finally, collecting all the estimates for  $L_i$ $(i=1,2,3)$, we obtain (\ref{eq:almostminimizer}).
\end{proof}

Now we obtain the decay estimate for the excess of the gradient of the minimizer of (\ref{F}) in the degenerate regime.
\begin{lem}\label{lem:degexcessdecay0}
Let $\gamma_0$ and $\alpha_2$ be given in Lemma~\ref{lem:gamma0} and \eqref{alpha2}, respectively.  For $\chi>0$ and $\gamma\in(0,\gamma_0)$, there exist $\tau\in(0,\frac14)$ and $\delta_3,\delta_4\in(0,1)$ depending on 
    $n, N, p, q, \alpha, [a]_{C^{0,\alpha}}, L, \nu, \beta_0,  \chi$, and $\gamma$ such that if $B_{2r}(x_0)\Subset\Omega$ satisfies (\ref{eq:smallness}) with $r\leq\frac{1}{2}$,
    \begin{equation}\label{eq:degeneratedecaay}
        \chi\fint_{B_{2r}(x_0)}\left|V_{H_{2r}^-}(Du)\right|^2\, dx\leq\fint_{B_{2r}(x_0)}\left|V_{H_{2r}^-}(Du)-(V_{H_{2r}^-}(Du))_{x_0,2r}\right|^2\, dx,
    \end{equation}
    \begin{equation}\label{eq:excessbound1}
        \fint_{B_{2r}(x_0)}\left|V_{H_{2r}^-}(Du)-(V_{H_{2r}^-}(Du))_{x_0,2r}\right|^2\, dx\leq \delta_3,
        \end{equation}
    and
    \begin{equation}\label{eq:radiusbound1}
       r^{\alpha_2}\leq\delta_4,
    \end{equation}
    then we have 
    \begin{equation*}
    \begin{aligned}
        \fint_{B_{2\tau r}(x_0)}&\left|V_{H^-_{2\tau r}}(Du)-(V_{H_{2\tau r}^-}(Du))_{x_0,2\tau r}\right|^2\, dx\\
        &\leq\tau^{2\gamma}\fint_{B_{2r}(x_0)}\left|V_{H_{2r}^-}(Du)-(V_{H_{2r}^-}(Du))_{x_0,2r}\right|^2\, dx,
    \end{aligned}
    \end{equation*}
 where   $H^-_\rho(t):=H^-_{B_{\rho}(x_0)}(t)$ for $\rho>0$.
\end{lem}

\begin{proof}
   We omit the dependence on $x_0$ from the notation. Note that by (\ref{eq:degeneratedecaay}) and (\ref{eq:excessbound1}),
   \begin{equation*}
       \fint_{B_{2r}}H_{2r}^-(|Du|)\, dx\leq\fint_{B_{2r}}|V_{H_{2r}^-}(Du)|^2\, dx\leq \chi^{-1}\delta_3.
   \end{equation*}
   
From Lemma~\ref{lem:almostminimizer} and \eqref{eq:Hoderintegrability},  we can  apply Lemma \ref{lem:minimizingapproximation} to $w=u$ and $\phi=H^-_{r}$. Furthermore, using \eqref{eq:Vtoshift} and the assumptions \eqref{eq:degeneratedecaay}\,--\,\eqref{eq:radiusbound1}, we obtain that
   \begin{equation*}
   \begin{aligned}
   &\fint_{B_{r}}|V_{H_{r}^-}(Du)-V_{H_{r}^-}(D\Bar{h})|^2\, dx\\
   & \leq c\left\{\frac{\left(\fint_{{2r}}H^-_{2r}(|Du|)\, dx\right)^{\sigma_0}}{H^-_{r}(\kappa)^{\sigma_0}}+r^{\alpha_2}+\delta\right\}^{\bar\gamma}\fint_{B_{2r}}H^-_{B_{2r}}(|Du|)\, dx\\
 & \leq C_0\left\{\bigg(\frac{\chi^{-1}\delta_3}{\kappa^{p\sigma_0}}\bigg)^{\bar\gamma}+\delta_4^{\bar\gamma} +\delta^{\bar\gamma}\right\} \fint_{B_{2r}}|V_{H_{2r}^-}(Du)|^2\, dx
   \end{aligned}
   \end{equation*}
for some $C_0\ge 1$. Here, $\bar h\in u+ W_0^{1,H_{r}^-}(B_{r},\R^N)$ is the unique minimizer of the energy $\int_{B_{r}} H_{r}^-(|Dv|)\, dx$ and $\bar\gamma:= \frac{\sigma_0p}{(1+\sigma_0)q+p}$. We first choose $\delta>0$ such that $C_0\delta^{\bar\gamma}\leq\frac{1}{2}\tau^{2\gamma_0+n}$; hence $\kappa$ is fixed. We next choose $\delta_3$ and $\delta_4$ so that
   \begin{equation}\label{delta3delta4}
       C_0 \bigg\{\bigg(\frac{\chi^{-1}\delta_3}{\kappa^{p\sigma_0}}\bigg)^{\bar\gamma}+\delta_4^{\bar\gamma}\bigg\}\leq\frac{1}{2}\tau^{2\gamma_0+n}.
   \end{equation}
   Then these choices of the parameters yield
   \begin{equation}\label{eq:ndegdecay1}
       \fint_{B_{r}}|V_{H_{r}^-}(Du)-V_{H_{r}^-}(D\Bar{h})|^2\, dx
       \leq \tau^{2\gamma_0+n}\fint_{B_{2r}}|V_{H_{2r}^-}(Du)|^2\, dx.
   \end{equation}
   
   We next observe from the elementary inequality $|\sqrt{1+t_1}-\sqrt{1+t_2}|^2\leq|t_1-t_2|$ for $t_1,t_2\geq0$ that for any $P\in\mathbb{R}^{N\times n}$,
   \begin{equation}\label{eq:Vdiff}
   \begin{aligned}
       |V_{H_{{2\tau r}}^-}(P)-V_{H_{r}^-}(P)|^2&\leq|P|^p\left|\sqrt{1+a_{2\tau r}^-\frac{q}{p}|P|^{q-p}}-\sqrt{1+a_{r}^-\frac{q}{p}|P|^{q-p}}\right|^2\\
       &\leq\frac{q}{p}(a_{{2\tau r}}^--a_{{r}}^-)|P|^q.
   \end{aligned}
   \end{equation}
   Thus, with (\ref{eq:qintegrabilityq}), (\ref{eq:Vsquare}), (\ref{eq:radiusbound1}) and the inequality $\delta_4\leq\tau^{2\gamma_0+n}$ (which follows from \eqref{delta3delta4}),
   \begin{equation}\label{eq:ndegdecay2}
       \begin{aligned}
           \fint_{B_{2\tau r}}|V_{H_{2\tau r}^-}(Du)-V_{H_{r}^-}(Du)|^2\, dx&\leq \frac{q}{p}\fint_{B_{2\tau r}}(a_{{2\tau r}}^--a_{{r}}^-)|Du|^q\, dx\\
           &\leq \frac{q}{2^np}\tau^{-n}\fint_{B_{r}}(a(x)-a_{{r}}^-)|Du|^q\, dx\\
           &\leq c\tau^{-n}\delta_4\fint_{B_{2r}}|V_{H_{2r}^-}(Du)|^2\, dx\\
           &\leq c\tau^{2\gamma_0}\fint_{B_{2r}}|V_{H_{2r}^-}(Du)|^2\, dx.
       \end{aligned}
   \end{equation}

  Therefore, using (\ref{eq:ndegdecay1}),  (\ref{eq:ndegdecay2}), Lemma \ref{lem:gamma0} with $\bar\phi=H_{r}^-$, and the inequality $\int_{B_{r}}H^-_{r}(|D\bar{h}|)\,dx \le  \int_{B_{r}}H^-_{r}(|Du|)\,dx$, we have  that
   \begin{equation*}
       \begin{aligned}
         &  \fint_{B_{2\tau r}}|V_{H_{{2\tau r}}^-}(Du)-(V_{H_{2\tau r}^-}(Du))_{{2\tau r}}|^2\,dx
           \leq\fint_{B_{2\tau r}}|V_{H_{{2\tau r}}^-}(Du)-(V_{H_{{r}}^-}(D\bar{h}))_{{2\tau r}}|^2\, dx\\
           &\quad\leq 4\bigg(\fint_{B_{2\tau r}}|V_{H_{{2\tau r}}^-}(Du)-V_{H_{{r}}^-}(Du)|^2\, dx +\fint_{B_{2\tau r}}|V_{H_{{r}}^-}(Du)-V_{H_{{r}}^-}(D\bar{h})|^2\, dx\\
           &\qquad\qquad+\fint_{B_{2\tau r}}|V_{H_{{r}}^-}(D\bar{h})-(V_{H_{{r}}^-}(D\bar{h}))_{{2\tau r}}|^2\, dx\bigg)\\
           &\quad\leq c\bigg(\tau^{2\gamma_0}\fint_{B_{2r}}|V_{H_{2r}^-}(Du)|^2\, dx +\tau^{-n}\fint_{B_{r}}|V_{H_{{r}}^-}(Du)-V_{H_{{r}}^-}(D\bar{h})|^2\, dx\\
           &\qquad\qquad+\tau^{2\gamma_0}\fint_{B_{r}}|V_{H_{{r}}^-}(D\bar{h})-(V_{H_{{r}}^-}(D\bar{h}))_{{r}}|^2\, dx\bigg)\\
           &\quad\leq c\tau^{2\gamma_0}\bigg(\fint_{B_{2r}}|V_{H_{2r}^-}(Du)|^2\, dx+\fint_{B_{r}}|V_{H_{r}^-}(Du)|^2\, dx\bigg)\\
           &\quad\leq \tilde c\tau^{2\gamma_0}\fint_{B_{2r}}|V_{H_{{2r}}^-}(Du)|^2\, dx
       \end{aligned}
   \end{equation*}
for some $\tilde c>0$, where in the last inequality we used \eqref{eq:Hoderintegrability}.
Finally, choosing $\tau>0$ such that $\tilde{c}\tau^{2\gamma_0}\chi^{-1}\leq\tau^{2\gamma}$, i.e., $\tau\leq (\chi/\tilde{c})^{\frac{1}{2(\gamma_0-\gamma)}}$, and using \eqref{eq:degeneratedecaay}, we obtain that
             \begin{equation*}
       \begin{aligned}
         &  \fint_{B_{2\tau r}}|V_{H_{2\tau r}^-}(Du)-(V_{H_{2\tau r}^-}(Du))_{2\tau r}|^2\,dx\\
           &\qquad \leq \tilde{c}\tau^{2\gamma_0}\chi^{-1}\fint_{B_{2r}}|V_{H_{2r}^-}(Du)-(V_{H_{2r}^-}(Du))_{{2r}}|^2\, dx\\
           &\qquad\leq \tau^{2\gamma}\fint_{B_{2r}}|V_{H_{2r}^-}(Du)-(V_{H_{2r}^-}(Du))_{2r}|^2\, dx.
       \end{aligned}
   \end{equation*}
\end{proof}

\section{Iteration} 

In this section, we prove our main result, Theorem~\ref{mnthm}, by means of an iteration argument. Let $u\in W^{1,p}(\Omega,\mathbb{R}^N)$ with $H(\cdot,|Du|)\in L^1(\Omega)$ be a minimizer of (\ref{F}), where the function $H$ in \eqref{H} satisfies \eqref{a} and  the function $f(x,P)$ in \eqref{F} satisfies \eqref{C1}--\eqref{C7}.  We begin by establishing an excess decay estimate in the nondegenerate regime.

\begin{lem}\label{lem:iteration}
    Let $B_{2R}(x_0)\Subset\Omega$ with $R\in(0,\frac{1}{4})$ satisfy (\ref{eq:smallness}) with $r$ replaced by $R$, and let $\beta\in(0,\frac{\alpha_3}{4}]$ with $\alpha_3$ given in \eqref{alpha3}. Then there exist $\delta_5,\delta_6\in(0,1)$ and $c>0$ depending on  $n, N, p, q, \alpha, [a]_{C^{0,\alpha}}, L, \nu, \beta_0$,  and $\beta$ such that if
    \begin{equation*}
        \fint_{B_R(x_0)}\left|V_{H_{R}^-}(Du)-(V_{H_{R}^-}(Du))_{x_0,R}\right|^2\, dx\leq\delta_5\fint_{B_R(x_0)}\left|V_{H_{R}^-}(Du)\right|^2\, dx
    \end{equation*}
    and
    \begin{equation}\label{eq:smallradius}
        R^\frac{\alpha_3}{2}\leq\delta_6,
    \end{equation}
    then we have 
    \begin{equation}\label{eq:ndegexcessdecay}
        \begin{aligned}
            &\fint_{B_r(x_0)}\left|V_{H_{r}^-}(Du)-(V_{H_{r}^-}(Du))_{x_0,r}\right|^2\, dx\\
            &\leq c\left(\frac{r}{R}\right)^{2\beta}\fint_{B_R(x_0)}\left|V_{H_{R}^-}(Du)-(V_{H_{R}^-}(Du))_{x_0,R}\right|^2\, dx +cr^{2\beta}\fint_{B_R(x_0)}\left|V_{H_{R}^-}(Du)\right|^2\, dx
        \end{aligned}
    \end{equation}
    for every $r\in(0,\frac{R}{8})$, where $H^-_\rho(t):= H^{-}_{B_\rho(x_0)}(t)$. 
    \end{lem}
    
\begin{proof} 

  We  omit the dependence on $x_0$ from the notation. Let us first choose parameters
    \begin{equation}\label{def:decaytauep}
        \tau:=\min\left\{\left(\frac{1}{2c^*}\right)^\frac{1}{1-\beta},\left(\frac{1}{16}\right)^\frac{1}{1-\beta}\right\}<\frac18\quad\text{ and }\quad\varepsilon:=\frac{\tau^{1+n+\beta}}{2c^*}
    \end{equation}
    where $c^*>0$ is a constant to be determined later. With this $\varepsilon$, we determine $\delta_1$ and $\delta_2$ from Lemma \ref{lem:ndegexcessdecay0}, and then choose
    \begin{equation}\label{def:delta56}
        \delta_5:=\min\left\{\delta_1,\frac{1}{8(1+\tau^{-n})},\frac{(\sqrt{2}-1)^2(1-\tau^\beta)^2\tau^n}{2}\right\}\quad\text{ and }\quad\delta_6:=\min\{\delta_2,\delta_5\}.
    \end{equation}
    
   Now,  we prove  the following inequalities by induction on $k\in\mathbb{N}_0$:
    \begin{equation}\label{eq:itid1}
        \fint_{B_{\tau^kR}}\left|V_k(Du)-(V_k(Du))_{\tau^kR}\right|^2\, dx\leq\tau^{2\beta k}\delta_5\fint_{B_{\tau^kR}}\left|V_k(Du)\right|^2\, dx
    \end{equation}
    \begin{equation}\label{eq:itid2}
        \begin{aligned}
            &\fint_{B_{\tau^kR}}\left|V_k(Du)-(V_k(Du))_{\tau^kR}\right|^2\, dx\\
            &\leq\tau^{(1+\beta)k}\fint_{B_{R}}\left|V_0(Du)-(V_0(Du))_{R}\right|^2\, dx+\frac{1-\tau^{(1-\beta)k}}{1-\tau^{1-\beta}}\left(\tau^kR\right)^{2\beta}\fint_{B_{R}}\left|V_0(Du)\right|^2\, dx
        \end{aligned}
    \end{equation}
    \begin{equation}\label{eq:itid3}
        \fint_{B_{\tau^kR}}\left|V_k(Du)\right|^2\, dx\leq2\fint_{B_{R}}\left|V_0(Du)\right|^2\, dx
    \end{equation}
    where $V_k:=V_{H_{\tau^kR}^-}$. We use the notations (\ref{eq:itid1})$_k$, (\ref{eq:itid2})$_k$, and (\ref{eq:itid3})$_k$ to denote the inequalities  (\ref{eq:itid1}), (\ref{eq:itid2}), and (\ref{eq:itid3}) with the specific $k\in\mathbb{N}_0$, respectively.  We first notice that these three inequalities hold for the case $k=0$ by the given conditions of the lemma.
    
    Assume (\ref{eq:itid1})$_h$, (\ref{eq:itid2})$_h$, and (\ref{eq:itid3})$_h$ hold for $h=0,1,2,\dots,k-1$ for some $k\in\mathbb{N}$. By (\ref{def:delta56}), the inequalities (\ref{eq:itid1})$_{k-1}$ and (\ref{eq:smallradius}) imply (\ref{eq:smalldecaycondition}) and (\ref{eq:smallradiuscondition}) respectively, with $r=\frac{\tau^{k-1}R}{2}$. Therefore applying Lemma \ref{lem:ndegexcessdecay0}, with $r=\frac{\tau^{k-1}R}{2}$ and $\tau$ replaced by $2\tau <\frac14$, and Jensen's inequality, we find
    \begin{equation*}
        \begin{aligned}
            &\fint_{B_{\tau^kR}}|V_k(Du)-(V_k(Du))_{\tau^kR}|^2\, dx\\
            &\leq c^*\tau^2\left(1+\frac{\varepsilon}{\tau^{n+2}}\right)\Bigg(\fint_{B_{\tau^{k-1}R}}|V_{k-1}(Du)-(V_{k-1}(Du))_{\tau^{k-1}R}|^2\, dx\\
            &\qquad\qquad\qquad\qquad\qquad\qquad\qquad+(\tau^{k-1}R)^\frac{\alpha_3}{2}\fint_{B_{\tau^{k-1}R}}|V_{k-1}(Du)|^2\, dx\Bigg)
        \end{aligned}
    \end{equation*}
    for some $c^*>0$. By (\ref{def:decaytauep}), $c^*\tau^{1-\beta}\leq\frac{1}{2}$ and $c^*\tau^{-\beta-n-2}\varepsilon=\frac{1}{2}$, we have
    \begin{equation*}
        c^*\tau^2\left(1+\frac{\varepsilon}{\tau^{n+2}}\right)=\tau^{1+\beta}(c^*\tau^{1-\beta}+c^*\tau^{-\beta-n-2}\varepsilon)\leq\tau^{1+\beta}. 
    \end{equation*}
    Thus since $\beta\le \frac{\alpha_3}{4}$, $\delta_6\leq\delta_5$ and $\tau^{1-\beta}\leq\frac{1}{16}$, using (\ref{eq:itid1})$_{k-1}$ and (\ref{eq:smallradius}), we see that
    \begin{equation}\label{eq:excessdecayinduction1}
        \begin{aligned}
            &\fint_{B_{\tau^kR}}|V_k(Du)-(V_k(Du))_{\tau^kR}|^2\, dx\\
            &\leq \tau^{1+\beta}\Bigg(\fint_{B_{\tau^{k-1}R}}|V_{k-1}(Du)-(V_{k-1}(Du))_{\tau^{k-1}R}|^2\, dx\\
            &\qquad\qquad\qquad\qquad+(\tau^{k-1}R)^\frac{\alpha_3}{2}\fint_{B_{\tau^{k-1}R}}|V_{k-1}(Du)|^2\, dx\Bigg)\\
            &\leq \tau^{1-\beta}\tau^{2\beta}\left(\tau^{2\beta(k-1)}\delta_5\fint_{B_{\tau^{k-1}R}}\left|V_{k-1}(Du)\right|^2\, dx+\tau^{2\beta(k-1)}\delta_6\fint_{B_{\tau^{k-1}R}}|V_{k-1}(Du)|^2\, dx\right)\\
            &\leq \frac{1}{8}\tau^{2\beta k}\delta_5\fint_{B_{\tau^{k-1}R}}\left|V_{k-1}(Du)\right|^2\, dx.
        \end{aligned}
    \end{equation}
    
    Moreover, (\ref{eq:itid1})$_{k-1}$ with the fact that $4(1+\tau^{-n})\delta_5\leq\frac{1}{2}$ from (\ref{def:delta56}) gives
    \begin{equation*}
        \begin{aligned}
            &\fint_{B_{\tau^{k-1}R}}\left|V_{k-1}(Du)\right|^2\, dx\\
            &\leq 4\fint_{B_{\tau^{k-1}R}}|V_{k-1}(Du)-(V_{k-1}(Du))_{\tau^{k-1}R}|^2\, dx\\
            &\qquad+4|(V_{k-1}(Du))_{\tau^{k-1}R}-(V_{k-1}(Du))_{\tau^{k}R}|^2+4\fint_{B_{\tau^{k}R}}\left|V_{k-1}(Du)\right|^2\, dx\\
            &\leq 4(1+\tau^{-n})\fint_{B_{\tau^{k-1}R}}|V_{k-1}(Du)-(V_{k-1}(Du))_{\tau^{k-1}R}|^2\, dx +4\fint_{B_{\tau^{k}R}}\left|V_{k-1}(Du)\right|^2\, dx\\
            &\leq 4(1+\tau^{-n})\delta_5\fint_{B_{\tau^{k-1}R}}|V_{k-1}(Du)|^2\, dx+4\fint_{B_{\tau^{k}R}}\left|V_{k-1}(Du)\right|^2\, dx\\
            &\leq \frac{1}{2}\fint_{B_{\tau^{k-1}R}}|V_{k-1}(Du)|^2\, dx+4\fint_{B_{\tau^{k}R}}\left|V_{k-1}(Du)\right|^2\, dx
        \end{aligned}
    \end{equation*}
    which implies that
    \begin{equation*}
        \fint_{B_{\tau^{k-1}R}}\left|V_{k-1}(Du)\right|^2\, dx\leq8\fint_{B_{\tau^{k}R}}\left|V_{k-1}(Du)\right|^2\, dx\leq8\fint_{B_{\tau^{k}R}}\left|V_{k}(Du)\right|^2\, dx,
    \end{equation*}
    where in the last inequality, we used the fact that $a_{\tau^{k-1}R}^-\leq a_{\tau^{k}R}^-$. Inserting this into (\ref{eq:excessdecayinduction1}), we have (\ref{eq:itid1})$_k$.\\
    We next prove (\ref{eq:itid2})$_k$ by using the first inequality of (\ref{eq:excessdecayinduction1}),  (\ref{eq:itid2})$_{k-1}$ and  (\ref{eq:itid3})$_{k-1}$ 
    \begin{equation*}
        \begin{aligned}
            &\fint_{B_{\tau^kR}}|V_k(Du)-(V_k(Du))_{\tau^kR}|^2\, dx\\
            &\leq \tau^{1+\beta}\Bigg(\fint_{B_{\tau^{k-1}R}}|V_{k-1}(Du)-(V_{k-1}(Du))_{\tau^{k-1}R}|^2\, dx\\
            &\qquad\qquad\qquad\qquad\qquad\qquad
            +(\tau^{k-1}R)^\frac{\alpha_3}{2}\fint_{B_{\tau^{k-1}R}}|V_{k-1}(Du)|^2\, dx\Bigg)\\
            &\leq \tau^{1+\beta}\fint_{B_{\tau^{k-1}R}}|V_{k-1}(Du)-(V_{k-1}(Du))_{\tau^{k-1}R}|^2\, dx\\
            &\qquad +\tau^{1-\beta}(\tau^kR)^{2\beta}\fint_{B_{\tau^{k-1}R}}|V_{k-1}(Du)|^2\, dx\\
            &\leq \tau^{(1+\beta)k}\fint_{B_{R}}|V_{0}(Du)-(V_{0}(Du))_{R}|^2\, dx\\
            &\quad +\tau^{1+\beta}\frac{1-\tau^{(1-\beta)(k-1)}}{1-\tau^{1-\beta}}(\tau^{k-1}R)^{2\beta}\fint_{B_R}|V_0(Du)|^2\, dx+(\tau^kR)^{2\beta}\fint_{B_R}|V_0(Du)|^2\, dx\\
            &\leq \tau^{(1+\beta)k}\fint_{B_{R}}|V_{0}(Du)-(V_{0}(Du))_{R}|^2\, dx+\frac{1-\tau^{(1-\beta)k}}{1-\tau^{1-\beta}}(\tau^{k}R)^{2\beta}\fint_{B_R}|V_0(Du)|^2\, dx,
        \end{aligned}
    \end{equation*}
    where we used the inequalities $\beta\le\frac{\alpha_3}{4}$ and that $2\tau^{1-\beta}<1$.
    
     Finally, (\ref{eq:itid1})$_h$ and (\ref{eq:itid3})$_h$, for $h=0,1,2,\dots,k-1$, and the fact that $\tau^{-\frac{n}{2}}(2\delta_5)^\frac{1}{2}\frac{1}{1-\tau^\beta}\leq\sqrt{2}-1$ from (\ref{def:delta56}) imply (\ref{eq:itid3})$_k$ as follows:
     \begin{equation*}
         \begin{aligned}
             &\left(\fint_{B_{\tau^kR}}|V_k(Du)|^2\, dx\right)^\frac{1}{2}\\
             &\leq\tau^{-\frac{n}{2}}\sum_{h=0}^{k-1}\left(\fint_{B_{\tau^hR}}|V_{h+1}(Du)-(V_h(Du))_{\tau^hR}|^2\, dx\right)^\frac{1}{2}+\left(\fint_{B_{R}}|V_0(Du)|^2\, dx\right)^\frac{1}{2}\\
             &\leq \tau^{-\frac{n}{2}}\delta_5^\frac{1}{2}\sum_{h=0}^{k-1}\tau^{\beta h}\left(\fint_{B_{\tau^hR}}|V_{h}(Du)|^2\, dx\right)^\frac{1}{2}+\left(\fint_{B_{R}}|V_0(Du)|^2\, dx\right)^\frac{1}{2}\\
             &\leq \left(\tau^{-\frac{n}{2}}(2\delta_5)^\frac{1}{2}\frac{1}{1-\tau^\beta}+1\right)\left(\fint_{B_{R}}|V_0(Du)|^2\, dx\right)^\frac{1}{2}\\
             &\leq \left(2\fint_{B_{R}}|V_0(Du)|^2\, dx\right)^\frac{1}{2}.
         \end{aligned}
     \end{equation*}

 We have proved the inequalities (\ref{eq:itid1}), (\ref{eq:itid2}), and (\ref{eq:itid3}). As the last step of the proof, let $r\in(0,\frac{R}{2})$ satisfy $\tau^{k+1}R\leq 2r<\tau^kR$ for some $k\geq0$. Then, since
     \begin{equation*}
         \begin{aligned}
             \fint_{B_r}|V_r^-(Du)-V_k(Du)|^2\, dx
             &\leq c\fint_{B_r}(a_r^--a_{\tau^kR}^-)|Du|^q\, dx\\
             &\leq c\fint_{B_r}(a(x)-a_{\tau^kR}^-)|Du|^q\, dx\\
             &\leq c\left(\frac{\tau^kR}{2r}\right)^n\fint_{B_{\frac{1}{2}\tau^kR}}(a(x)-a_{\tau^kR}^-)|Du|^q\, dx\\
             &\leq c\tau^{-n}(\tau^kR)^{\alpha_0}\fint_{B_{\tau^kR}}|V_k(Du)|^2\, dx,
         \end{aligned}
     \end{equation*}
     where we used (\ref{eq:Vdiff}) and (\ref{eq:qintegrabilityq}), we find
     \begin{equation*}
         \begin{aligned}
             &\fint_{B_r}|V_r^-(Du)-(V_r^-(Du))_r|^2\, dx\\
             &\leq \fint_{B_r}|V_r^-(Du)-(V_k(Du))_{\tau^kR}|^2\, dx\\
             &\leq 2\fint_{B_r}|V_r^-(Du)-V_k(Du)|^2\, dx+2\fint_{B_r}|V_k(Du)-(V_k(Du))_{\tau^kR}|^2\, dx\\
             &\le c \underbrace{\tau^{-n}(\tau^kR)^{\alpha_0}\fint_{B_{\tau^kR}}|V_k(Du)|^2\, dx}_{=:I_1}+c\underbrace{\tau^{-n}\fint_{B_{\tau^kR}}|V_k(Du)-(V_k(Du))_{\tau^kR}|^2\, dx}_{=:I_2}.
         \end{aligned}
     \end{equation*}
     Since $\tau^{k+1}R<2r<1$, we use (\ref{eq:itid3}) to have
     \begin{equation*}
     \begin{aligned}
         I_1<\tau^{-n-\alpha_0}(\tau^{k+1}R)^\frac{\alpha_0}{2}\fint_{B_{\tau^kR}}|V_k(Du)|^2\, dx&\leq c\tau^{-n-\alpha_0}r^\frac{\alpha_0}{2}\fint_{B_{\tau^kR}}|V_k(Du)|^2\, dx\\
         &\leq c\tau^{-n-\alpha_0}r^\frac{\alpha_0}{2}\fint_{B_{R}}|V_0(Du)|^2\, dx.
     \end{aligned}
     \end{equation*}
     On the other hand, by (\ref{eq:itid2}) and using the fact that $1+\beta>2\beta$, we have
     \begin{equation*}
         \begin{aligned}
             I_2&\leq\tau^{-n+(1+\beta)k}\fint_{B_R}|V_0(Du)-(V_0(Du))_R|^2\, dx\\
             &\qquad+\tau^{-n}\frac{1-\tau^{(1-\beta)k}}{1-\tau^{1-\beta}}(\tau^kR)^{2\beta}\fint_{B_R}|V_0(Du)|^2\, dx\\
             &\leq c\tau^{-n-1-\beta}\left(\frac{r}{R}\right)^{2\beta}\fint_{B_R}|V_0(Du)-(V_0(Du))_R|^2\, dx\\
             &\qquad+c\frac{\tau^{-n}}{1-\tau^{1-\beta}}\left(\frac{r}{\tau}\right)^{2\beta}\fint_{B_R}|V_0(Du)|^2\, dx.
         \end{aligned}
     \end{equation*}
     Since $2\beta\le \frac{\alpha_3}{2}$ and $\tau$ is a fixed constant given in \eqref{def:decaytauep}, we finally have (\ref{eq:ndegexcessdecay}).
\end{proof}

Finally, we prove our main theorem.

\subsection*{Proof of Theorem~\ref{mnthm}}\ \\
     Let $\beta:=\min\left\{\frac{\gamma_0}{2},\frac{\alpha_3}{4}\right\}$, where $\gamma_0$ and $\alpha_3$ are given in Lemma~\ref{lem:gamma0} and \eqref{alpha3}, respectively. Applying Lemma \ref{lem:iteration} with this $\beta$, we obtain $\delta_5$ and $\delta_6$. We then choose $\chi=\delta_5$ and $\gamma=\beta$ in Lemma~\ref{lem:degexcessdecay0}, hence we also determine $\delta_3$, $\delta_4$ and $\tau$.

     Now choose any $x_1\in\Omega$ satisfying
     \begin{equation*}
         \liminf_{r\to0+}\fint_{B_r(x_1)}\left|V_{H_{B_r(x_1)}^-}(Du)-(V_{H_{B_r(x_1)}^-}(Du))_{x_1,r}\right|^2\, dx=0
     \end{equation*}
     and
     \begin{equation*}
         M:=\limsup_{r\to0+}\fint_{B_r(x_1)}\left|V_{H_{B_r(x_1)}^-}(Du)\right|^2\, dx<+\infty.
     \end{equation*}
     Fix $\Omega'\Subset\Omega$ such that $x_1\in\Omega'$. Then there exists $r_0\in(0,\frac{1}{4})$ such that (\ref{eq:smallness}) holds for every $r\in(0,r_0)$. Moreover, we may find $R_0\in(0,\frac{r_0}{2})$ such that $B_{2R_0}(x_1)\subset\Omega'$ with
     \begin{equation}\label{eq:R0condition}
         \begin{dcases}
             R_0^\frac{\alpha_3}{2}\leq\min\left\{\frac{\delta_3}{4(M+1)},\delta_6,\delta_4\right\},\\
             \fint_{B_{R_0}(x_1)}\left|V_{H_{B_{R_0}(x_1)}^-}(Du)-(V_{H_{B_{R_0}(x_1)}^-}(Du))_{x_1,R_0}\right|^2\, dx\leq\frac{\delta_3}{4},\\
             \fint_{B_{R_0}(x_1)}\left|V_{H_{B_{R_0}(x_1)}^-}(Du)\right|^2\, dx\leq M+1.
         \end{dcases}
     \end{equation}
     Since the integrals are continuous with respect to the translation of the domain of integration, there exists $R_1\in(0,R_0)$ such that for every $x_0\in B_{R_1}(x_1)$, we have
     \begin{equation}\label{eq:R1condition}
         \begin{dcases}
             \fint_{B_{R_0}(x_0)}\left|V_{H_{B_{R_0}(x_0)}^-}(Du)-(V_{H_{B_{R_0}(x_0)}^-}(Du))_{x_0,R_0}\right|^2\, dx\leq\frac{\delta_3}{2}\\
             \fint_{B_{R_0}(x_0)}\left|V_{H_{B_{R_0}(x_0)}^-}(Du)\right|^2\, dx\leq 2(M+1).
         \end{dcases}
     \end{equation}
     Now we fix an arbitrary $x_0\in B_{R_1}(x_1)$ and write $V_k:=V_{H_{B_{\tau^kR_0}(x_0)}^-}$. In what follows, we omit $x_0$ and write $H_\rho^\pm:=H_{B_\rho(x_0)}^\pm$ for $\rho>0$.

     We first suppose that 
     \begin{equation}\label{eq:degsupposure}
         \delta_5\fint_{B_{\tau^kR_0}}|V_k(Du)|^2\, dx\leq\fint_{B_{\tau^kR_0}}|V_k(Du)-(V_k(Du))_{\tau^kR_0}|^2\, dx\quad\text{ for all }k\geq0.
     \end{equation}
     In view of (\ref{eq:R0condition}) and (\ref{eq:R1condition}), we may use Lemma \ref{lem:degexcessdecay0} to find
     \begin{equation}\label{eq:degbound}
         \begin{aligned}
             &\fint_{B_{\tau^kR_0}}|V_k(Du)-(V_k(Du))_{\tau^kR_0}|^2\, dx\\
             &\leq\tau^{2\beta}\fint_{B_{\tau^{k-1}R_0}}|V_{k-1}(Du)-(V_{k-1}(Du))_{\tau^{k-1}R_0}|^2\, dx\\
             &\leq\cdots\leq\tau^{2k\beta}\fint_{B_{R_0}}|V_0(Du)-(V_0(Du))_{R_0}|^2\, dx\leq\tau^{2k\beta}\frac{\delta_3}{2}
         \end{aligned}
     \end{equation}
     holds for every $k\in\mathbb{N}_0$. Let $r\in(0,\frac{R_0}{2})$ be given and $k\in\mathbb{N}_0$ satisfy $\tau^{k+1}R_0\leq 2r<\tau^kR_0$. Then since
     \begin{equation*}
         \begin{aligned}
             \fint_{B_r}|V_{H_r^-}(Du)-V_k(Du)|^2\, dx&\leq c\fint_{B_r}(a_r^--a_{\tau^kR_0}^-)|Du|^q\, dx\\
             &\leq c\fint_{B_r}(a(x)-a_{\tau^kR_0}^-)|Du|^q\, dx\\
             &\leq c\left(\frac{\tau^kR_0}{2r}\right)^n\fint_{B_{\frac{1}{2}\tau^kR_0}}(a(x)-a_{\tau^kR_0}^-)|Du|^q\, dx\\
             &\leq c\tau^{-n}(\tau^kR_0)^{\alpha_0}\fint_{B_{\tau^kR_0}}|V_k(Du)|^2\, dx\\
             &\leq c\tau^{-n}\fint_{B_{\tau^kR_0}}|V_k(Du)|^2\, dx,
         \end{aligned}
     \end{equation*}
     we have, with (\ref{eq:degsupposure}) and (\ref{eq:degbound}),
     \begin{equation*}
         \begin{aligned}
             &\fint_{B_r}|V_{H_r^-}(Du)-(V_{H_r^-}(Du))_r|^2\, dx\\
             &\leq\fint_{B_r}|V_{H_r^-}(Du)-(V_k(Du))_{\tau^kR_0}|^2\, dx\\
             &\leq2\fint_{B_r}|V_{H_r^-}(Du)-V_k(Du)|^2\, dx+2\fint_{B_r}|V_k(Du)-(V_k(Du))_{\tau^kR_0}|^2\, dx\\
             &\leq c\tau^{-n}\fint_{B_{\tau^kR_0}}|V_k(Du)|^2\, dx+c\tau^{-n}\fint_{B_{\tau^kR_0}}|V_k(Du)-(V_k(Du))_{\tau^kR_0}|^2\, dx\\
             &\leq c(\delta_5^{-1}+1)\tau^{-n}\fint_{B_{\tau^kR_0}}|V_k(Du)-(V_k(Du))_{\tau^kR_0}|^2\, dx\\
             &\leq c\delta_3(\delta_5^{-1}+1)\tau^{-n}\tau^{2k\beta}\leq c\delta_3(\delta_5^{-1}+1)\tau^{-n}\left(\frac{r}{\tau R_0}\right)^{2\beta}.
         \end{aligned}
     \end{equation*}
     Therefore we obtain
     \begin{equation}\label{eq:degresult}
         \fint_{B_r}\frac{|V_{H_r^-}(Du)-(V_{H_r^-}(Du))_r|^2}{r^{2\beta}}\, dx\leq c\frac{\delta_3(\delta_5^{-1}+1)}{\tau^{n+2\beta}R_0^{2\beta}}.
     \end{equation}

     We next suppose (\ref{eq:degsupposure}) fails. Then there is $k_0\in\mathbb{N}_0$ such that 
     \begin{equation}\label{eq:ndegfails}
         \delta_5\fint_{B_{\tau^kR_0}}|V_k(Du)|^2\, dx\leq\fint_{B_{\tau^kR_0}}|V_k(Du)-(V_k(Du))_{\tau^kR_0}|^2\, dx
     \end{equation}
     for every $k=0,\dots,k_0-1$ (for $k_0=0$, we neglect this condition) and 
     \begin{equation*}
         \fint_{B_{\tau^{k_0}R_0}}|V_{k_0}(Du)-(V_{k_0}(Du))_{\tau^{k_0}R_0}|^2\, dx<\delta_5\fint_{B_{\tau^{k_0}R_0}}|V_{k_0}(Du)|^2\, dx.
     \end{equation*}
     If $k_0=0$, Lemma \ref{lem:iteration} and (\ref{eq:R1condition}) are applied to find that for any $r\in(0,\frac{R_0}{2})$,
     \begin{equation*}
         \begin{aligned}
             \fint_{B_r}&|V_{H_r^-}(Du)-(V_{H_r^-}(Du))_r|^2\, dx\\
             &\leq c\left(\frac{r}{R_0}\right)^{2\beta}\fint_{B_{R_0}}|V_0(Du)-(V_0(Du))_0|^2\, dx+cr^{2\beta}\fint_{B_{R_0}}|V_0(Du)|^2\, dx\\
             &\leq c\delta_3\left(\frac{r}{R_0}\right)^{2\beta}+cr^{2\beta}(M+1),
         \end{aligned}
     \end{equation*}
     and hence
     \begin{equation*}
         \fint_{B_r}\frac{|V_{H_r^-}(Du)-(V_{H_r^-}(Du))_r|^2}{r^{2\beta}}\, dx\leq c\left(\frac{\delta_3}{R_0^{2\beta}}+M+1\right).
     \end{equation*}
     Now we consider the case when $k_0\geq1$. Note that if $2r\in[\tau^{k_0}R_0,R_0)$, by (\ref{eq:ndegfails}), we have (\ref{eq:degresult}). If $2r\in(0,\tau^{k_0}R_0)$, by Lemma \ref{lem:iteration} with $R=\tau^{k_0}R_0$,
     \begin{equation*}
         \begin{aligned}
             \fint_{B_r}&|V_{H_r^-}(Du)-(V_{H_r^-}(Du))_r|^2\, dx\\
             &\leq c\left(\frac{r}{\tau^{k_0}R_0}\right)^{2\beta}\fint_{B_{\tau^{k_0}R_0}}|V_{k_0}(Du)-(V_{k_0}(Du))_{\tau^{k_0}R_0}|^2\, dx+cr^{2\beta}\fint_{B_{\tau^{k_0}R_0}}|V_{k_0}(Du)|^2\, dx\\
             &\leq c\frac{\delta_3}{\tau^{n+2\beta}}\left(\frac{r}{ R_0}\right)^{2\beta}(\delta_5^{-1}+1)+cr^{2\beta}\fint_{B_{\tau^{k_0}R_0}}|V_{k_0}(Du)|^2\, dx,
         \end{aligned}
     \end{equation*}
     where in the last inequality, we used (\ref{eq:degresult}) with $r=\tau^{k_0}R_0$. Moreover,
     by (\ref{eq:Vdiff},) (\ref{eq:qintegrabilityq}), and (\ref{eq:Vsquare}), with the inequality $\tau^{k_0-1}R_0<1$,
     \begin{equation*}
         \begin{aligned}
             \fint_{B_{\tau^{k_0}R_0}}|V_{k_0}(Du)-V_{k_0-1}(Du)|^2\, dx&\leq c\fint_{B_{\tau^{k_0}R_0}}(a_{\tau^{k_0}R_0}^--a_{\tau^{k_0-1}R_0}^-)|Du|^q\, dx\\
             &\leq c\tau^{-n}\fint_{B_{\frac{1}{2}\tau^{k_0-1}R_0}}(a(x)-a_{\tau^{k_0-1}R_0}^-)|Du|^q\, dx\\
             &\leq c\tau^{-n}(\tau^{k_0-1}R_0)^{\alpha_0}\fint_{B_{\tau^{k_0-1}R_0}}|V_{k_0-1}(Du)|^2\, dx\\
             &\leq c\tau^{-n}\fint_{B_{\tau^{k_0-1}R_0}}|V_{k_0-1}(Du)|^2\, dx.
         \end{aligned}
     \end{equation*}
     Using (\ref{eq:ndegfails}) and (\ref{eq:degbound}) with $k=k_0-1$, this further implies that
     \begin{equation*}
         \begin{aligned}
             \fint_{B_{\tau^{k_0}R_0}}|V_{k_0}(Du)|^2\, dx&\leq2\fint_{B_{\tau^{k_0}R_0}}|V_{k_0}(Du)-V_{k_0-1}(Du)|^2\, dx+2\fint_{B_{\tau^{k_0}R_0}}|V_{k_0-1}(Du)|^2\, dx\\
             &\leq c\tau^{-n}\fint_{B_{\tau^{k_0-1}R_0}}|V_{k_0-1}(Du)|^2\, dx\\
             &\leq c\tau^{-n}\delta_5^{-1}\fint_{B_{\tau^{k_0-1}R_0}}|V_{k_0-1}(Du)-(V_{k_0-1}(Du))_{\tau^{k_0-1}R_0}|^2\, dx\\
             &\leq c\tau^{-n}\delta_5^{-1}\delta_3.
         \end{aligned}
     \end{equation*}
     Therefore for any $r\in(0,R_0)$, we have
     \begin{equation*}
         \fint_{B_r}\frac{|V_{H_r^-}(Du)-(V_{H_r^-}(Du))_r|^2}{r^{2\beta}}\, dx\leq c\frac{\delta_3}{\tau^{n+2\beta}R_0^{2\beta}}(\delta_5^{-1}+1)+c\tau^{-n}\delta_5^{-1}\delta_3.
     \end{equation*}

     Consequently, we have the inequality
     \begin{equation*}
         \fint_{B_r(x_0)}\frac{|V_{H_{B_r(x_0)}^-}(Du)-(V_{H_{B_r(x_0)}^-}(Du))_{x_0,r}|^2}{r^{2\beta}}\, dx\leq C
     \end{equation*}
     for every ball $B_r(x_0)$ with $x_0\in B_{R_1}(x_1) $ and any $r\in(0, R_0)$. We further notice that from (\ref{eq:Vtoshift}) and (\ref{eq:shiftislinear}),
     \begin{equation*}
         \begin{aligned}
             |V_{H_r^-}(P_1)-V_{H_r^-}(P_2)|^2&\sim H_{|P_2|}(x_r^-,|P_1-P_2|)\\
             &\gtrsim |V_p(P_1)-V_p(P_2)|^2+a_r^-|V_q(P_1)-V_q(P_2)|^2\\
             &\geq |V_p(P_1)-V_p(P_2)|^2,
         \end{aligned}
     \end{equation*}
     which, with (\ref{eq:Vmean}) and the previous consequence, implies
     \begin{equation*}
         \fint_{B_r(x_0)}\frac{|V_{p}(Du)-(V_{p}(Du))_{x_0,r}|^2}{r^{2\beta}}\, dx\leq C.
     \end{equation*}
     Therefore, we finally have that $V_p(Du)\in C^{0,\beta}(B_{R_1}(x_1),\R^{N\times n})$, which implies that $Du\in C^{0,\tilde \beta}(B_{R_1}(x_1),\R^{N\times n})$ for some $\tilde \beta\in (0,1)$, see, e.g. \cite[Lemma 2.10]{DSV09}. This completes the proof. \qed

\section*{Acknowledgement}

J. Ok was supported by the
National Research Foundation of Korea funded by the Korean Government (NRF-2022R1C1C1004523).

\subsection*{Conflicts of Interest}  The authors declare no conflict of interest.

\subsection*{Data availability}
No data was used for the research described in the article.


\begin{thebibliography}{99}

\bibitem{A68}
F. J. Almgren:
Existence and regularity almost everywhere of solutions to elliptic variational problems among surfaces of varying topological type and singularity structure,
Ann. of Math. (2) 87 (1968) 321–391.

\bibitem{AF87}
E. Acerbi and N. Fusco:
A regularity theorem for minimizers of quasiconvex integrals,
Arch. Rational Mech. Anal. 99 (1987), no. 3, 261–281.


\bibitem{B12}
V. B\"ogelein:
Partial regularity for minimizers of discontinuous quasi-convex integrals with degeneracy,
J. Differential Equations 252 (2) (2012), 1052–1100.


\bibitem{BCM18}
P. Baroni, M. Colombo and G. Mingione:
Regularity for general functionals with double phase,
Calc. Var. Partial Differ.Equ. 57 (2) (2018) 62.


\bibitem{CM15.1}
M. Colombo and G. Mingione:
Regularity for double phase variational problems,
Arch. Ration. Mech. Anal. 215 (2)(2015) 443--496.

\bibitem{CM15.2}
M. Colombo and G. Mingione:
Bounded minimisers of double phase variational integrals,
Arch. Ration. Mech. Anal.218 (1) (2015) 219--273.


\bibitem{CO20}
P. Celada and J. Ok:
Partial regularity for non-autonomous degenerate quasi-convex functionals with general growth,
Nonlinear Anal. 194 (2020) 111473.

\bibitem{G61}
E. De Giorgi:
Frontiere orientate di misura minima, Seminario di Matematica della Scuola Normale Superiore di Pisa, 1960-61
Editrice Tecnico Scientifica, Pisa 1961 57 pp


\bibitem{DGG00}
F. Duzaar, A. Gastel and J. Grotowski:
Partial regularity for almost minimizers of quasiconvex functionals,
SIAM J. Math. Anal. 32 (2000) 665–687.

\bibitem{DG00}
F. Duzaar and J. Grotowski:
Optimal interior partial regularity for nonlinear elliptic systems: the method of A-harmonic approximation,
Manuscripta Math. 103 (2000) 267–298.

\bibitem{DS02}
F. Duzaar and K. Steffen:
Optimal interior and boundary regularity for almost minimizers to elliptic variational integrals, 
J. Reine Angew. Math. 546 (2002) 73–138.

\bibitem{DK02}
F. Duzaar and M. Kronz:
Regularity of $\omega$-minimizers of quasi-convex variational integrals with polynomial growth, 
Differential Geom. Appl. 17 (2002),no. 2-3, 139–152.

\bibitem{DM04}
F. Duzaar and G. Mingione:
The $p$-harmonic approximation and the regularity of $p$-harmonic maps,
Calc. Var. Partial Differential Equations, 20 (2004), pp. 235–256.

\bibitem{DM041}
F. Duzaar and G. Mingione:
Regularity for degenerate elliptic problems via $p$-harmonic approximation,
Ann. Inst. H. Poincaré C Anal. Non Linéaire 21 (2004), no. 5, 735--766.



\bibitem{DE08}
L. Diening and F. Ettwein:
Fractional estimates for non-differentiable elliptic systems with general growth,
Forum Math. 20 (3) (2008) 523--556.

\bibitem{DSV09}
L. Diening, B. Stroffolini and A. Verde:
Everywhere regularity of functionals with $\varphi$-growth,
Manuscr. Math. 129 (4)(2009) 449--481.

\bibitem{DLSV12}
L. Diening, D. Lengeler, B. Stroffolini and A. Verde:
Partial regularity for minimizers of quasiconvex functionals with general growth,
SIAM J. Math. Anal. 44 (5) (2012) 3594--3616.

\bibitem{DSV12}
L. Diening, B. Stroffolini, A. Verde, The $\phi$-harmonic approximation and the regularity of $\phi$-harmonic maps,
J. Differ. Equ. 253 (7) (2012) 1943--1958.

\bibitem{DKS12}
L. Diening, P. Kaplický and S. Schwarzacher:
BMO estimates for the p-Laplacian, Nonlinear Anal. 75 (2) (2012) 637--650.

\bibitem{D19}
C. De Filippis:
On the regularity of the $\omega$-minima of $\varphi$-functionals,
Nonlinear Anal. 194 (2020), 111464, 25 pp.

\bibitem{ELM04}
L. Esposito, F. Leonetti and G. Mingione:
Sharp regularity for functionals with $(p,q)$ growth,
J. Differential Equations 204 (2004), no. 1, 5--55.

\bibitem{E86}
L. C. Evans:
Quasiconvexity and partial regularity in the calculus of variations,
Arch. Rational Mech. Anal. 95 (1986), no. 3, 227–252.

\bibitem{FMM04}
I. Fonseca, J. Mal\'y, G. Mingione:
Scalar minimizers with fractal singular sets,
Arch. Ration. Mech. Anal. 172 (2) (2004) 295--307.

\bibitem{FM08}
M. Foss and G. Mingione:
Partial continuity for elliptic problems,
Ann. Inst. H. Poincar\'e Anal. Non Lin\'eaire 25 (2008), no. 3, 471--503.

\bibitem{Gh73}
F. W. Gehring:
The $L_p$-integrability of the partial derivatives of a quasiconformal mapping,
Acta Math. 130 (1973) 265--277.

\bibitem{GM68}
E. Giusti and M. Miranda:
Sulla regolaritá delle soluzioni deboli di una classe di sistemi ellittici quasi-lineari, 
Arch. Ration. Mech. Anal. 31 (1968/1969) 173–184.

\bibitem{Gia83}
M. Giaquinta:
multiple integrals in the calculus of variations and nonlinear elliptic systems,
Princeton University Press, Princeton (1983).

\bibitem{G02}
J.F. Grotowski:
Boundary regularity for nonlinear elliptic systems,
Calc. Var. 15 (2002) 353–388.

\bibitem{Giusti} E.\ Giusti: 
\emph{Direct Methods in the Calculus of Variations}, 
World Scientific, Singapore, 2003.

\bibitem{G16}
F. Giannetti:
A $C^{1,\alpha}$ partial regularity result for integral functionals with $p(x)$-growth condition, Adv. Calc. Var. 9 (2016), no. 4, 395--407.

\bibitem{GPRT17}
F. Giannetti, A. Passarelli di Napoli, M.A. Ragusa and A. Tachikawa:
Atsushi Partial regularity for minimizers of a class of non autonomous functionals with nonstandard growth,
Calc. Var. Partial Differential Equations 56 (2017), no. 6, Art. 153, 29 pp.

\bibitem{GPT17}
F. Giannetti, A. Passarelli di Napoli and A. Tachikawa:
Atsushi Partial regularity results for non-autonomous functionals with $\Psi$-growth conditions
Ann. Mat. Pura Appl. (4) 196 (2017), no. 6, 2147–2165.

\bibitem{H13}
J. Habermann:
Partial regularity for nonlinear elliptic systems with continuous growth exponent,
Ann. Mat. Pura Appl. (4) 192 (2013), no. 3, 475–527.

\bibitem{HO23}
P. H\"ast\"or and J. Ok: 
 Regularity theory for non-autonomous problems with a  priori assumptions,
 Calc. Var. Partial Differential Equations 62 (2023), no. 6, Article No. 251p.

\bibitem{KM18}
T. Kuusi and G. Mingione:
Vectorial nonlinear potential theory,
J. Eur. Math. Soc. (JEMS) 20 (2018), no. 4, 929--1004.

\bibitem{M52}
C. B. Morrey:
Quasi-convexity and the lower semicontinuity of multiple integrals,
Pacific J. Math. 2 (1952) 25--53.

\bibitem{M67}
C. B. Morrey:
Partial regularity results for non-linear elliptic systems,
J. Math. Mech. 17 (1967/1968) 649--670.

\bibitem{N78}
I. Ne\v cas:
An example of a nonsmooth solution of a nonlinear elliptic system with analytic coecients and a smoothness condition,
Proceedings of an AllUnion Conference on Partial Differential Equations (Moscow State Univ.,Moscow,1976), pp. 174--177, Moskov. Gos. Univ., Mekh.-Mat. Fakul’tet, Moscow, 1978.

\bibitem{O16}
J. Ok:
Partial continuity for a class of elliptic systems with non-standard growth,
Electron. J. Differential Equations 2016, Paper No. 323, 24 pp.

\bibitem{O17}
J. Ok:
Regularity of $\omega$-minimizers for a class of functionals with non-standard growth,
Calc. Var. Partial Differ. Equ.56 (2017) 48.

\bibitem{OSS25}
J. Ok, G. Scilla and B. Stroffolini:
Partial regularity for degenerate systems of double phase type,
J. Differ.Equ. 432 (2025) 113207.

\bibitem{O18}
J. Ok:
Partial Hölder regularity for elliptic systems with non-standard growth,
J. Funct. Anal. 274 (3) (2018)723--768.

\bibitem{O18B}
J. Ok:
Partial regularity for general systems of double phase type with continuous coefficients,
Nonlinear Anal. 177 (2018), part B, 673–698.


\bibitem{S09}
T. Schmidt:
A simple partial regularity proof for minimizers of variational integrals,
NoDEA Nonlinear Differ. Eqns. Appl. 16(1). 109--129 (2009).

\bibitem{S83}
L. Simon:
Lectures on Geometric Measure Theory,
Proc. CMA, vol. 3, ANU Canberra, 1983.

\bibitem{S96}
L. Simon:
Theorems on Regularity and Singularity of Energy Minimizing Maps,
Birkhäuser-Verlag, Basel, 1996.

\bibitem{SY02}
V. Sver\'ak and X. Yan:
Non-Lipschitz minimizers of smooth uniformly convex functionals,
Proc. Natl. Acad. Sci. USA 99 (2002), no. 24, 15269–15276.

\bibitem{SS23}
G. Scilla and  B. Stroffolini:
Partial regularity for steady double phase fluids,
Math. Eng. 5 (5) (2023) 1--47.

\bibitem{U77}
K. Uhlenbeck:
Regularity for a class of non-linear elliptic systems,
Acta Math. 138 (3--4) (1977) 219--240.

\bibitem{Z86}
V. V. Zhikov:
Averaging of functionals of the calculus of variations and elasticity theory,
Izv. Akad. Nauk SSSR, Ser.Mat. 50 (4) (1986) 675--710, 877.

\bibitem{Z95}
V. V. Zhikov:
On Lavrentiev’s phenomenon,
Russ. J. Math. Phys. 3 (2) (1995) 249--269.

\bibitem{Z97}
V. V. Zhikov:
On some variational problems,
Russ. J. Math. Phys. 5 (1) (1997) 105--116 (1998).


\end{thebibliography}
\end{document}